\documentclass[12pt]{amsart}

\usepackage{latexsym}
\usepackage{amsxtra}
\usepackage{amsmath}
\usepackage{amsfonts}
\usepackage{amssymb}
\usepackage{blkarray}
\usepackage{multirow}
\usepackage{pgf,tikz}
\usepackage{mathrsfs}
\usetikzlibrary{arrows}

\newtheorem{theorem}{Theorem}[section]
\newtheorem{lemma}[theorem]{Lemma}
\newtheorem{corollary}[theorem]{Corollary}
\newtheorem{proposition}[theorem]{Proposition}

\newtheorem{conjecture}[theorem]{Conjecture}

\theoremstyle{definition}
\newtheorem{definition}[theorem]{Definition}

\theoremstyle{remark}

\theoremstyle{notation}

\theoremstyle{claim}

\numberwithin{equation}{section}

\begin{document}
\title{Templates for Binary Matroids}

\author{Kevin Grace}
\address{Department of Mathematics\\
Louisiana State University\\
Baton Rouge, Louisiana}
\email{kgrace3@lsu.edu}

\author{Stefan H. M. van Zwam}
\address{Department of Mathematics\\
Louisiana State University\\
Baton Rouge, Louisiana}
\email{svanzwam@math.lsu.edu}
\thanks{The first author was supported by a Huel D. Perkins Fellowship from the Louisiana State University Graduate School. The second author was supported by the National Science Foundation, grant 1500343.}

\date{}

\begin{abstract}
A \textit{binary frame template} is a device for creating binary matroids from graphic or cographic matroids. Such matroids are said to \textit{conform} or \textit{coconform} to the template. We introduce a preorder on these templates and determine the nontrivial templates that are minimal with respect to this order. As an application of our main result, we determine the eventual growth rates of certain minor-closed classes of binary matroids, including the class of binary matroids with no minor isomorphic to $PG(3,2)$. Our main result applies to all highly-connected matroids in a class, not just those of maximum size. As a second application, we characterize the highly-connected 1-flowing matroids.
\end{abstract}

\maketitle
\section{Introduction}
\label{introduction}
Geelen, Gerards, and Whittle ~\cite{ggw15} recently announced a structure theorem describing the highly connected members of any proper minor-closed class of matroids representable over a given finite field. In this paper we study some consequences of their result. To state a first, rough version of their result, we need the following definitions.

A matroid $M$ is \textit{vertically $k$-connected} if, for each partition $(X,Y)$ of the ground set of $M$ with $r(X)+r(Y)-r(M)<k-1$, either $X$ or $Y$ is spanning. We denote the unique prime subfield of $\mathbb{F}$ by $\mathbb{F}_{\textnormal{prime}}$. We say that a matroid $M_2$ is a \textit{rank-$(\leq t)$ perturbation} of a matroid $M_1$ if there exist matrices $A_1$ and $A_2$ over $\mathbb{F}$ such that $r(M(A_1-A_2))\leq t$ and such that $M_1\cong M(A_1)$ and $M_2\cong M(A_2)$. 

We now restate ~\cite[Theorem 3.3]{ggw15}. Its proof is forthcoming in a future paper by Geelen, Gerards, and Whittle.
\begin{theorem}\label{ggw3.3}
Let $\mathbb{F}$ be a finite field and let $m_0$ be a positive integer. Then there exist $k,n,t\in\mathbb{Z}_+$ such that, if $M$ is a matroid representable over $\mathbb{F}$ such that $M$ or $M^*$ is vertically $k$-connected and such that $M$ has an $M(K_n)$-minor but no $PG(m_0-1,\mathbb{F}_{\textnormal{prime}})$-minor, then $M$ is a rank-$(\leq t)$ perturbation of a frame matroid representable over $\mathbb{F}$.
\end{theorem}

Let us consider a very simple example of a rank-1 perturbation. Let $A_1$ be the binary matrix
\[\begin{bmatrix}
1&0&0&0&1&1&1&0&0&0\\
0&1&0&0&1&0&0&1&1&0\\
0&0&1&0&0&1&0&1&0&1\\
0&0&0&1&0&0&1&0&1&1\\
\end{bmatrix},
\]
and let $A_2$ be the binary matrix
\[\begin{bmatrix}
0&1&1&1&1&1&1&0&0&0\\
1&0&1&1&1&0&0&1&1&0\\
0&0&1&0&0&1&0&1&0&1\\
0&0&0&1&0&0&1&0&1&1\\
\end{bmatrix}.
\]
Note that $A_2$ is the result of adding the rank-1 matrix
\[\begin{bmatrix}
1&1&1&1&0&0&0&0&0&0\\
1&1&1&1&0&0&0&0&0&0\\
0&0&0&0&0&0&0&0&0&0\\
0&0&0&0&0&0&0&0&0&0\\
\end{bmatrix}
\] to $A_1$. Therefore, the vector matroid $M(A_2)$ is a rank-1 perturbation of $M(A_1)$.

Theorem ~\ref{ggw3.3} is essentially a simplified version of a much more complex structure theorem ~\cite[Theorem 4.2]{ggw15}. Geelen, Gerards, and Whittle introduced the concept of a \textit{template} as a tool to capture much of this complexity.

Our focus in this paper is on the binary case. Roughly speaking, a binary frame template can be thought of as a recipe for constructing a representable matroid from a graphic or cographic matroid. A matroid constructed in this way is said to \textit{conform} or \textit{coconform} to the template. 

In the example above, we may think of $M(A_2)$ as the matroid obtained from the vector matroid of the following matrix by contracting the element indexing the final column. Note that the large submatrix on the bottom left is $A_1$:
\[
\left[
\begin{array}{@{}cccccccccc|c@{}}
1&1&1&1&0&0&0&0&0&0&1\\
\hline
1&0&0&0&1&1&1&0&0&0&1\\
0&1&0&0&1&0&0&1&1&0&1\\
0&0&1&0&0&1&0&1&0&1&0\\
0&0&0&1&0&0&1&0&1&1&0\\
\end{array}
\right]
\]
In fact, for any matrix $A$ of the following form, where $v$ and  $w$ are arbitrary binary vectors, the matroid $M(A)/c$ conforms to the template $\Phi_{CX}$, which we will define in Section ~\ref{Reducing a Template}:
\begin{center}
\begin{tabular}{|c|c|}
\multicolumn{1}{c}{}&\multicolumn{1}{c}{$c$}\\
\hline
$v$&1\\
\hline
&\\
incidence matrix of a graph&$w$\\
&\\
\hline
\end{tabular}
\end{center}

Let $\mathcal{M}(\Phi)$ denote the set of matroids representable over a field $\mathbb{F}$ that conform to a frame template $\Phi$.  Theorem ~\ref{ggwframe} below is a slight modification of ~\cite[Theorem 4.2]{ggw15}. The modification is explained in Section ~\ref{Preliminaries}.
\begin{theorem}\label{ggwframe}
Let $\mathbb{F}$ be a finite field, let $m$ be a positive integer, and let $\mathcal{M}$ be a minor-closed class of matroids representable over $\mathbb{F}$. Then there exist $k,l\in \mathbb{Z}_+$ and frame templates $\Phi_1,\dots, \Phi_s, \Psi_1,\dots, \Psi_t$ such that
\begin{itemize}
\item $\mathcal{M}$ contains each of the classes $\mathcal{M}(\Phi_1),\dots,\mathcal{M}(\Phi_s)$,
\item $\mathcal{M}$ contains the duals of the matroids in each of the classes $\mathcal{M}(\Psi_1),\dots,\mathcal{M}(\Psi_t)$, and
\item if $M$ is a simple vertically $k$-connected member of $\mathcal{M}$ with at least $l$ elements and with no $PG(m-1,\mathbb{F}_{\textnormal{prime}})$ minor, then either $M$ is a member of at least one of the classes $\mathcal{M}(\Phi_1),\dots,\mathcal{M}(\Phi_s)$, or $M^*$ is a member of at least one of the classes $\mathcal{M}(\Psi_1),\dots,\mathcal{M}(\Psi_t)$.
\end{itemize}
\end{theorem}

Our contribution is to shed some light on how these templates are related to each other. We define a preorder on the set of frame templates. Our main result, Theorem ~\ref{minimal}, is a list of nontrivial binary frame templates that are minimal with respect to this preorder.

One application of this result involves growth rates of minor-closed classes of binary matroids. The \textit{growth rate function} of a minor-closed class $\mathcal{M}$ is the function whose value at an integer $r\geq0$ is given by the maximum number of elements in a simple matroid in $\mathcal{M}$ of rank at most $r$. We prove that a minor-closed class of binary matroids has a growth rate that is eventually equal to the growth rate of the class of graphic matroids if and only if it contains all graphic matroids but does not contain the class of matroids conforming to a certain template. The class of matroids conforming to this template is exactly the class of matroids having an even-cycle representation with a blocking pair. Geelen and Nelson also proved this result in ~\cite{gn15}. We also prove the following theorem. Here, $\mathcal{EX}(F)$ denotes the class of binary matroids with no $F$-minor. If $f$ and $g$ are functions, we write $f(r)\approx g(r)$ if $f(r)=g(r)$ for all but finitely many $r$.

\begin{theorem}
\label{EXPG32}
 The growth rate function for $\mathcal{EX}(PG(3,2))$ is \[h_{\mathcal{EX}(PG(3,2))}\approx r^2-r+1.\]
\end{theorem}
Note that $r^2-r+1$ is the growth rate of the class of even-cycle matroids.

Our main result goes beyond growth rates because it gives information about all highly-connected matroids in a minor-closed class, not just the maximum-sized matroids. This is illustrated by our second application, involving 1-flowing matroids. The 1-flowing property is a generalization of the max-flow min-cut property of graphs. We prove the following.
\begin{theorem}
\label{1flowing}
 There exist $k,l\in\mathbb{Z}_+$ such that every simple, vertically $k$-connected, 1-flowing matroid with at least $l$ elements is either graphic or cographic.
\end{theorem}

We use templates to study a minor-closed class $\mathcal{M}$ by describing the highly-connected matroids in the class. This analysis follows a certain pattern:
\begin{enumerate}
 \item Find a matroid $N$ not in $\mathcal{M}$.
\item Find all templates such that $N$ is not a minor of any matroid conforming to that template.
\item If all matroids conforming to these templates are in $\mathcal{M}$, then the analysis is complete.
\item Otherwise, repeat Step (1).
\end{enumerate}

From the definition of conforming to a template, which we will give in Section ~\ref{Preliminaries}, it will not be difficult to see that for each binary frame template $\Phi$, there are integers $t_1$ and $t_2$ such that every matroid conforming to $\Phi$ is a rank-$(\leq t_1)$ perturbation of a graphic matroid and every matroid coconforming to $\Phi$ is a rank-$(\leq t_2)$ perturbation of a cographic matroid. Thus, by Theorem ~\ref{ggwframe}, the highly connected matroids in a minor-closed class of binary matroids are ``close'' to being graphic or cographic. In this regard, the work regarding templates resembles work done by Robertson and Seymour concerning minor-closed classes of graphs. In Theorem 1.3 of ~\cite{rs03}, Robertson and Seymour showed that highly-connected graphs in a minor-closed class are in some sense ``close'' to being embeddable in some surface.

Section ~\ref{Preliminaries} of this paper repeats the necessary definitions found in ~\cite{ggw15}. In Section ~\ref{Reducing a Template}, we prove our main result, as well as giving some machinery leading up to it. Section ~\ref{Growth Rates} applies our result to growth rates of minor-closed classes of binary matroids, and in Section ~\ref{1-flowing Matroids}, we prove Theorem ~\ref{1flowing}.

\section{Preliminaries}
\label{Preliminaries}

We repeat here several definitions concerning highly connected matroids which can be found in  Geelen, Gerards, and Whittle ~\cite{ggw15}. Although the results found in ~\cite{ggw15} are technically about matrices rather than matroids, it suffices for our purposes to state the results in terms of their immediate matroid consequences.

Let $A$ be a matrix over a field $\mathbb{F}$. Then $A$ is a \textit{frame matrix} if each column of $A$ has at most two nonzero entries. We let $\mathbb{F}^{\times}$ denote the multiplicative group of $\mathbb{F}$. Let $\Gamma$ be a subgroup of $\mathbb{F}^{\times}$. A $\Gamma$-frame matrix is a frame matrix $A$ such that:
\begin{itemize}
 \item Each column of $A$ with a nonzero entry contains a 1.
 \item If a column of $A$ has a second nonzero entry, then that entry is $-\gamma$ for some $\gamma\in\Gamma$.
\end{itemize}
In the case where $\Gamma$ is the multiplicative group of one element, a matrix is a $\Gamma$-frame matrix if and only if it is the signed incidence matrix of a graph, with possibly a row removed. In particular, a binary matroid is graphic if and only if it can be represented by a matrix over $\mathrm{GF}(2)$ in which no column has more than two nonzero entries.

To facilitate the description of their structure theorem, Geelen, Gerards, and Whittle capture capture much of the complexity with the concept of a ``template.'' Let $\mathbb{F}$ be a finite field. A \textit{frame template} over $\mathbb{F}$ is a tuple $\Phi=(\Gamma,C,X,Y_0,Y_1,A_1,\Delta,\Lambda)$ such that the following hold\footnote{The authors of ~\cite{ggw15} divided our set $X$ into two separate sets which they called $X$ and $D$. Their set $X$ can be absorbed into $Y_0$, therefore we omit it.}:
\begin{itemize}
 \item [(i)] $\Gamma$ is a subgroup of $\mathbb{F}^{\times}$.
 \item [(ii)] $C$, $X$, $Y_0$ and $Y_1$ are disjoint finite sets.
 \item [(iii)] $A_1\in \mathbb{F}^{X\times (C\cup Y_0\cup Y_1)}$.
 \item [(iv)] $\Lambda$ is a subgroup of the additive group of $\mathbb{F}^X$ and is closed under scaling by elements of $\Gamma$.
 \item [(v)] $\Delta$ is a subgroup of the additive group of $\mathbb{F}^{C\cup Y_0 \cup Y_1}$ and is closed under scaling by elements of $\Gamma$.
\end{itemize}

Let $\Phi=(\Gamma,C,X,Y_0,Y_1,A_1,\Delta,\Lambda)$ be a frame template. Let $B$ and $E$ be finite sets, and let $A'\in\mathbb{F}^{B\times E}$. We say that $A'$ \textit{respects} $\Phi$ if the following hold:
\begin{itemize}
 \item [(i)] $X\subseteq B$ and $C, Y_0, Y_1\subseteq E$.
 \item [(ii)] $A'[X, C\cup Y_0\cup Y_1]=A_1$.
 \item [(iii)] There exists a set $Z\subseteq E-(C\cup Y_0\cup Y_1)$ such that $A'[X,Z]=0$, each column of $A'[B-X,Z]$ is a unit vector, and $A'[B-X, E-(C\cup Y_0\cup Y_1\cup Z)]$ is a $\Gamma$-frame matrix.
 \item [(iv)] Each column of $A'[X,E-(C\cup Y_0\cup Y_1\cup Z)]$ is contained in $\Lambda$.
 \item [(v)] Each row of $A'[B-X, C\cup Y_0\cup Y_1]$ is contained in $\Delta$.
\end{itemize}

Figure \ref{fig:A'} shows the structure of $A'$.

\begin{figure}
\begin{center}
\begin{tabular}{ r|c|c|ccc| }
\multicolumn{2}{c}{}&\multicolumn{1}{c}{$Z$}&\multicolumn{1}{c}{$Y_0$}&\multicolumn{1}{c}{$Y_1$}&\multicolumn{1}{c}{$C$}\\
\cline{2-6}
&&&&&\\
$X$&columns from $\Lambda$&$0$&&$A_1$&\\
&&&&&\\
\cline{2-6}
&\multirow{5}{*}{$\Gamma$-frame matrix}&\multirow{5}{*}{unit columns}&\multicolumn{3}{c|}{\multirow{5}{4em}{rows from  $\Delta$}}\\
&&&&&\\
&&&&&\\
&&&&&\\
&&&&&\\
\cline{2-6}
\end{tabular}
\end{center}
\caption{}
  \label{fig:A'}
\end{figure}

Suppose that $A'$ respects $\Phi$ and that $Z$ satisfies (iii) above. Now suppose that $A\in \mathbb{F}^{B\times E}$ satisfies the following conditions:
\begin{itemize}
\item [(i)] $A[B,E-Z]=A'[B,E-Z]$
\item [(ii)] For each $i\in Z$ there exists $j\in Y_1$ such that the $i$-th column of $A$ is the sum of the $i$-th and the $j$-th columns of $A'$.
\end{itemize}
We say that any such matrix \textit{conforms} to $\Phi$.

Let $M$ be a matroid representable over $\mathbb{F}$. We say that $M$ \textit{conforms} to $\Phi$ if there is a matrix $A$ that conforms to $\Phi$ such that $M$ is isomorphic to $M(A)/C\backslash Y_1$.

Let $\mathcal{M}(\Phi)$ denote the set of matroids representable over $\mathbb{F}$ that conform to $\Phi$. Recall that a matroid $M$ is \textit{vertically $k$-connected} if, for each partition $(X,Y)$ of the ground set of $M$ with $r(X)+r(Y)-r(M)<k-1$, either $X$ or $Y$ is spanning. We denote the unique prime subfield of $\mathbb{F}$ by $\mathbb{F}_{\textnormal{prime}}$. Geelen, Gerards, and Whittle will prove Theorem ~\ref{ggwframe} in a future paper. This theorem is actually a slight modification of the theorem found in ~\cite{ggw15}. In that paper, there is no mention of the requirement that a matroid have size at least $l$. However, Geelen (personal communication) has stated that this is necessary to ensure that adding a finite number of matroids to the class $\mathcal{M}$ does not add any templates to the list $\Phi_1,\dots, \Phi_s, \Psi_1,\dots, \Psi_t$.

Although the term \textit{coconform} does not appear in ~\cite{ggw15}, we define it in the following obvious way.

\begin{definition}
 A matroid $M$ \textit{coconforms} to a template $\Phi$ if its dual $M^*$ conforms to $\Phi$.
\end{definition}

To simplify the proofs in this paper, it will be helpful to expand the concept of conforming slightly.

\begin{definition}
 \label{virtual}
Let $A'$ be a matrix that respects $\Phi$, as defined above, except that we allow columns of $A'[B-X,Z]$ to be either unit columns or zero columns. Let $A$ be a matrix that is constructed from $A'$ as described above. Thus, $A[B,E-Z]=A'[B,E-Z]$, and for each $i\in Z$ there exists $j\in Y_1$ such that the $i$-th column of $A$ is the sum of the $i$-th and the $j$-th columns of $A'$. Let $M$ be isomorphic to $M(A)/C\backslash Y_1$. We say that $A$ and $M$ \textit{virtually conform} to $\Phi$ and that $A'$ \textit{virtually respects} $\Phi$. If $M^*$ virtually conforms to $\Phi$, we say that $M$ \textit{virtually coconforms} to $\Phi$.
\end{definition}

We will denote the set of matroids representable over $\mathbb{F}$ that virtually conform to $\Phi$ by $\mathcal{M}_v(\Phi)$ and the set of matroids representable over $\mathbb{F}$ that virtually coconform to $\Phi$ by $\mathcal{M}^*_v(\Phi)$.

The following notation will be used throughout this paper. We denote an empty matrix by $[\emptyset]$. We denote a group of one element by $\{0\}$ or $\{1\}$, if it is an additive or multiplicative group, respectively. If $S'$ is a subset of a set $S$ and $G$ is a subgroup of the additive group $\mathbb{F}^S$, we denote by $G|S'$ the projection of $G$ into $\mathbb{F}^{S'}$. Similarly, if $\bar{x}\in G$, we denote the projection of $\bar{x}$ into $\mathbb{F}^{S'}$ by $\bar{x}|S'$.

Unexplained notation and terminology will generally follow Oxley ~\cite{o11}. One exception is that we denote the vector matroid of a matrix $A$ by $M(A)$, rather than $M[A]$.

\section{Reducing a Template}
\label{Reducing a Template}

In this section, we will introduce reductions and show that every template reduces to one of several basic templates.

Since templates are used to study minor-closed classes of matroids, a natural question to ask is whether the set of matroids conforming to a particular template is minor-closed. The answer is no, in general. For example, if $|Y_0|=1$, then a matroid conforms to the following binary frame template if and only if it is a graphic matroid with a loop:
\[(\{1\},\emptyset,\emptyset,Y_0,\emptyset,[\emptyset],\{0\},\{0\}).\]
Clearly, this is not a minor-closed class.

Another question to ask is whether there might be some sort of minor relationship between a pair of templates, where every matroid conforming to one template is a minor of a matroid conforming to the other. These questions motivate the following discussion.

\begin{definition}
 A \textit{reduction} is an operation on a frame template $\Phi$ that produces a frame template $\Phi'$ such that $\mathcal{M}(\Phi')\subseteq \mathcal{M}(\Phi)$.
\end{definition}

\begin{proposition} \label{reductions}
The following operations are reductions on a frame template $\Phi$:
\begin{itemize}
\item[(1)] Replace $\Gamma$ with a proper subgroup.
\item[(2)] Replace $\Lambda$ with a proper subgroup closed under multiplication by elements from $\Gamma$.
\item[(3)] Replace $\Delta$ with a proper subgroup closed under multiplication by elements from $\Gamma$.
\item[(4)] Remove an element $y$ from $Y_1$. (More precisely, replace $A_1$ with $A_1[X, Y_0\cup (Y_1-y)\cup C]$ and replace $\Delta$ with $\Delta|(Y_0\cup (Y_1-y)\cup C)$.
\item[(5)] For all matrices $A'$ respecting $\Phi$, perform an elementary row operation on $A'[X, E]$. (Note that this alters the matrix $A_1$ and performs a change of basis on $\Lambda$.)
\item[(6)] If there is some element $x\in X$ such that, for every matrix $A'$ respecting $\Phi$, we have that $A'[\{x\},E]$ is a zero row vector, remove $x$ from $X$. (This simply has the effect of removing a zero row from every matrix conforming to $\Phi$.)
\item[(7)] Let $c\in C$ be such that $A_1[X,\{c\}]$ is a unit column whose nonzero entry is in the row indexed by $x\in X$, and let either $\lambda_x=0$ for each $\lambda\in\Lambda$ or $\delta_c=0$ for each $\delta\in\Delta$. Let $\Delta'$ be the result of adding $-\delta_cA_1[\{x\},Y_0\cup Y_1\cup C]$ to each element $\delta\in\Delta$. Replace $\Delta$ with $\Delta'$, and then remove $c$ from $C$ and $d$ from $D$. (More precisely, replace $A_1$ with $A_1[X-x, Y_0\cup Y_1\cup (C-c)]$, replace $\Lambda$ with $\Lambda|(X-x)$, and replace $\Delta$ with $\Delta'|(Y_0\cup Y_1\cup (C-c))$.)
\item[(8)] Let $c\in C$ be such that $A_1[X,\{c\}]$ is a zero column and $\delta_c=0$ for all $\delta\in\Delta$. Then remove $c$ from $C$. (More precisely, replace $A_1$ with $A_1[X, Y_0\cup Y_1\cup (C-c)]$, and replace $\Delta$ with $\Delta|(Y_0\cup Y_1\cup (C-c))$.)
\end{itemize}
\end{proposition}

\begin{proof}
Let $\Phi'$ be the template that results from performing one of operations (1)-(8) on $\Phi$.

For (1)-(3), every matrix $A'$ respecting $\Phi'$ also respects $\Phi$.

For (4), let $A'$ be a matrix respecting $\Phi'$, and let $M$ be the matroid $M(A)/C\backslash Y_1$, where $A$ is a matrix conforming to $\Phi'$ that has been constructed from $A'$ respecting $\Phi'$ as described in Section ~\ref{Preliminaries}. Since $Y_1$ is deleted to produce $M$, the only effect of $Y_1$ on $M$ is that  for each $i\in Z$ there exists $j\in Y_1$ such that the $i$-th column of $A$ is the sum of the $i$-th and the $j$-th columns of $A'$. But each $j\in Y_1$ in the template $\Phi'$ is also contained in $Y_1$ in the template $\Phi$. Therefore, $A$ conforms to $\Phi$, as does $M$.

For (5) and (6), elementary row operations and removing zero rows produce isomorphic matroids.

Operations (7)  and (8) have the effect of contracting $c$ from $M(A)\backslash Y_1$ for every matrix $A$ conforming to $\Phi$. Since all of $C$ is contracted to produce a matroid $M$ conforming to $\Phi$, the matroids we produce by performing either of these operations still conform to $\Phi$.
\end{proof}

For $i\in\{1,\dots, 8\}$, we call operation $(i)$ above a \textit{reduction of type $i$}.

The operations listed in the definition below are not reductions as defined above, but we continue the numbering from Proposition ~\ref{reductions} for ease of reference.

\begin{definition}
\label{weaklyconforming}
A template $\Phi'$ is a \textit{template minor} of $\Phi$ if $\Phi'$ is obtained from $\Phi$ by repeatedly performing the following operations:
\begin{itemize}
\item[(9)] Performing a reduction of type 1-8 on $\Phi$.
\item[(10)] Removing an element $y$ from $Y_0$, replacing $A_1$ with $A_1[X,(Y_0-y)\cup Y_1\cup C]$, and replacing $\Delta$ with $\Delta|((Y_0-y)\cup Y_1\cup C)$. (This has the effect of deleting $y$ from every matroid conforming to $\Phi$.)
\item[(11)] Let $x\in X$ with $\lambda_x=0$ for every $\lambda\in\Lambda$, and let $y\in Y_0$ be such that $(A_1)_{x,y}\neq0$. Then contract $y$ from every matroid conforming to $\Phi$. (More precisely, perform row operations on $A_1$ so that $A_1[X, \{y\}]$ is a unit column with $(A_1)_{x,y}=1$. Then replace every element $\delta\in\Delta$ with the row vector $-\delta_y A_1[\{x\}, Y_0\cup Y_1\cup C]+\delta$. This induces a group homomorphism $\Delta\rightarrow\Delta'$, where $\Delta'$ is also a subgroup of the additive group of $\mathbb{F}^{C\cup Y_0 \cup Y_1}$ and is closed under scaling by elements of $\Gamma$. Finally, replace $A_1$ with $A_1[X-x,(Y_0-y)\cup Y_1\cup C]$, project $\Lambda$ into $\mathbb{F}^{X-x}$, and project $\Delta'$ into $\mathbb{F}^{(Y_0-y)\cup Y_1\cup C}$. The resulting groups play the roles of $\Lambda$ and $\Delta$, respectively in $\Phi'$.)
\item[(12)] Let $y\in Y_0$ with $\delta_y=0$ for every $\delta\in\Delta$. Then contract $y$ from every matroid conforming to $\Phi$. (More precisely, if $A_1[X, \{y\}]$ is a zero vector, this is the same as simply removing $y$ from $Y_0$. Otherwise, choose some $x\in X$ such that $(A_1)_{x,y}\neq0$. Then for every matrix $A'$ that respects $\Phi$, perform row operations so that $A_1[X,\{y\}]$ is a unit column with $(A_1)_{x,y}=1$. This induces a group isomorphism $\Lambda\rightarrow\Lambda'$ where $\Lambda'$ is also a subgroup of the additive group of $\mathbb{F}^X$ and is closed under scaling by elements of $\Gamma$. Finally, replace $A_1$ with $A_1[X-x,(Y_0-y)\cup Y_1\cup C]$, project $\Lambda'$ into $\mathbb{F}^{X-x}$, and project $\Delta$ into $\mathbb{F}^{(Y_0-y)\cup Y_1\cup C}$. The resulting groups play the roles of $\Lambda$ and $\Delta$, respectively in $\Phi'$.)
\end{itemize}
\end{definition}

Let $\Phi'$ be a template minor of $\Phi$, and let $A'$ be a matrix that virtually respects $\Phi'$. Let $A$ be a matrix that virtually conforms to $\Phi'$, and let $M$ be a matroid that virtually conforms to $\Phi'$. We say that $A'$ \textit{weakly respects} $\Phi$ and that $A$ and $M$ \textit{weakly conform} to $\Phi$. Let $\mathcal{M}_w(\Phi)$ denote the set of matroids representable over $\mathbb{F}$ that weakly conform to $\Phi$, and let $\mathcal{M}^*_w(\Phi)$ denote the set of matroids representable over $\mathbb{F}$ whose duals weakly conform to $\Phi$. If $M\in\mathcal{M}^*_w(\Phi)$, we say that $M$ \textit{weakly coconforms} to $\Phi$.

\begin{lemma} \label{minor}
 If a matroid $M$ weakly conforms to a template $\Phi$, then $M$ is a minor of a matroid that conforms to $\Phi$.
\end{lemma}

\begin{proof}
Let $\Phi'$ be a template minor of $\Phi$. As we can see from Definition ~\ref{weaklyconforming}, every matroid $M$ weakly conforming to $\Phi'$ is a minor of a matroid virtually conforming to $\Phi$. It remains to analyze the case where $M$ virtually conforms to $\Phi$; so $M$ is isomorphic to $M(K)/C\backslash Y_1$, where $K$ is built from a matrix $K'$ that virtually respects $\Phi$. Consider the following matrix $A'$ obtained from $K'$ by adding a row $r$ and a column $c$.

\begin{center}
\begin{tabular}{ r|c|c|c|c|ccc| }
\multicolumn{1}{c}{}&\multicolumn{1}{c}{$c$}&\multicolumn{1}{c}{}&\multicolumn{2}{c}{$Z$}&\multicolumn{1}{c}{$Y_0$}&\multicolumn{1}{c}{$Y_1$}&\multicolumn{1}{c}{$C$}\\
\cline{2-8}
&&&\multicolumn{2}{|c|}{}&&&\\
$X$&0&columns from $\Lambda$&\multicolumn{2}{|c|}{0}&&$A_1$&\\
&&&\multicolumn{2}{|c|}{}&&&\\
\cline{2-8}
&\multirow{5}{*}{0}&\multirow{5}{*}{$\Gamma$-frame matrix}&\multirow{5}{*}{0}&\multirow{5}{*}{unit columns}&\multicolumn{3}{c|}{\multirow{5}{4em}{rows from $\Delta$}}\\
&&&&&&&\\
&&&&&&&\\
&&&&&&&\\
&&&&&&&\\
\cline{2-8}
$r$&1&0&$1\cdots1$&0&\multicolumn{3}{c|}{0}\\
\cline{2-8}
\end{tabular}
\end{center}

From $A'$, we can obtain a matrix $A$ conforming to $\Phi$ such that $M$ is isomorphic to $M(A)/C\backslash Y_1/c$. So $M$ is a minor of a matroid conforming to $\Phi$.
\end{proof}

An easy consequence of Lemma ~\ref{minor} is that Theorem ~\ref{ggwframe}, which deals with minor-closed classes, can be restated in terms of weak conforming. 

\begin{corollary} \label{weakframe}
 Let $\mathbb{F}$ be a finite field, let $m$ be a positive integer, and let $\mathcal{M}$ be a minor-closed class of matroids representable over $\mathbb{F}$. Then there exist $k,l\in \mathbb{Z}_+$ and frame templates $\Phi_1,\dots, \Phi_s, \Psi_1,\dots, \Psi_t$ such that
\begin{itemize}
\item $\mathcal{M}$ contains each of the classes $\mathcal{M}_w(\Phi_1),\dots,\mathcal{M}_w(\Phi_s)$,
\item $\mathcal{M}$ contains the duals of the matroids in each of the classes $\mathcal{M}_w(\Psi_1)$,$\dots$,$\mathcal{M}_w(\Psi_t)$, and
\item if $M$ is a simple vertically $k$-connected member of $\mathcal{M}$ with at least $l$ elements and with no $PG(m-1,\mathbb{F}_{\textnormal{prime}})$ minor, then either $M$ is a member of at least one of the classes \linebreak $\mathcal{M}_v(\Phi_1),\dots,\mathcal{M}_v(\Phi_s)$ or $M^*$ is a member of at least one of the classes $\mathcal{M}_v(\Psi_1),\dots,\mathcal{M}_v(\Psi_t)$.
\end{itemize}
\end{corollary}

\begin{proof}
Let  $\Phi_1,\dots, \Phi_s, \Psi_1,\dots, \Psi_t$ be the templates whose existence is implied by Theorem ~\ref{ggwframe}. For $\Phi\in\{\Phi_1,\dots, \Phi_s\}$, Lemma ~\ref{minor} implies that any matroid $M\in \mathcal{M}_w(\Phi)$ is a minor of a matroid $N\in\mathcal{M}(\Phi)$. Since $\mathcal{M}$ contains $\mathcal{M}(\Phi)$ and is minor-closed, $\mathcal{M}$ contains $\mathcal{M}_w(\Phi)$ as well. Similarly, $\mathcal{M}$  contains the duals of the matroids in each of the classes $\mathcal{M}_w(\Psi_1),\dots,\mathcal{M}_w(\Psi_t)$. The third condition above holds since every matroid conforming to a template also virtually conforms to it.
\end{proof}

If $\mathcal{M}_{w}(\Phi)=\mathcal{M}_{w}(\Phi')$, we say that $\Phi$ is \textit{equivalent} to $\Phi'$ and write $\Phi\sim\Phi'$. It is clear that $\sim$ is indeed an equivalence relation.

\begin{definition}
Let $T_{\mathbb{F}}$ be the set of all frame templates over $\mathbb{F}$. We define a preorder $\preceq$ on $T_{\mathbb{F}}$ as follows. We say $\Phi\preceq\Phi'$ if $\mathcal{M}_w(\Phi)\subseteq\mathcal{M}_w(\Phi')$. This is indeed a preorder since reflexivity and transitivity follow from the subset relation. We may obtain a partial order by considering equivalence classes of templates, with equivalence as defined above. However, the templates themselves, rather than equivalence classes, are the objects we work with in this paper.
\end{definition}

Let $\Phi_0$ be the frame template with all groups trivial and all sets empty. We call this template the \textit{trivial template}. In general, we say that a template $\Phi$ is \textit{trivial} if $\Phi\preceq\Phi_0$. It is easy to see that for any template $\Phi$, we have $\Phi_0\preceq \Phi$. Therefore, if $\Phi\preceq\Phi_0$, then actually $\Phi\sim\Phi_0$.

Our desire is to find a collection of minimal nontrivial templates. For the remainder of this paper, we restrict our attention to binary frame templates: those frame templates where $\mathbb{F}=\mathrm{GF}(2)$ and $\Gamma$ is the group of one element.

\begin{definition}
\leavevmode
\begin{itemize}
\item Let $\Phi_C$  be the template with all groups trivial and all sets empty except that $|C|=1$ and $\Delta\cong\mathbb{Z}/2\mathbb{Z}$.
\item Let $\Phi_X$  be the template with all groups trivial and all sets empty except that $|X|=1$ and $\Lambda\cong\mathbb{Z}/2\mathbb{Z}$.
\item Let $\Phi_{Y_0}$  be the template with all groups trivial and all sets empty except that $|Y_0|=1$ and $\Delta\cong\mathbb{Z}/2\mathbb{Z}$. 
\item Let $\Phi_{CX}$ be the template with $Y_0=Y_1=\emptyset$, with $|C|=|X|=1$, with $\Delta\cong\Lambda\cong\mathbb{Z}/2\mathbb{Z}$, with $\Gamma$ trivial, and with $A_1=[1]$.
\item Let $\Phi_{Y_1}$ be the template with all groups trivial, with $C=Y_0=\emptyset$, with $|Y_1|=3$ and $|X|=2$, and with $A_1=
\begin{bmatrix}
1& 0 &1\\
0& 1 & 1
\end{bmatrix}$.
\end{itemize}
\end{definition}

It is not too difficult to see that the Fano matroid $F_7$ virtually conforms to each of $\Phi_C$, $\Phi_X$, $\Phi_{CX}$, $\Phi_{Y_0}$, and $\Phi_{Y_1}$. Therefore, these templates are nontrivial. In fact, one can see that $\mathcal{M}(\Phi_{Y_0})$ is the set of graft matroids, that $\mathcal{M}(\Phi_C)$ is the class of matroids obtained by closing the set of graft matroids under minors, and that $\mathcal{M}(\Phi_X)$ is the class of even-cycle matroids. In Lemma ~\ref{Y1minors}, we will show that $\mathcal{M}_v(\Phi_{Y_1})$ is the class of matroids having an even-cycle representation with a blocking pair.

Our goal in defining reductions and weak conforming was essentially to perform operations on matrices while leaving the $\Gamma$-frame submatrix intact. The following lemma does not contribute to that goal; so we will only make occasional use of it.

\begin{lemma}
\label{YCD}
 The following relations hold:
\begin{itemize}
\item[(1)] $\Phi_{Y_1}\preceq\Phi_X$
\item[(2)] $\Phi_{Y_1}\preceq\Phi_C$
\item[(3)] $\Phi_{Y_0}\preceq\Phi_C$
\item[(4)] $\Phi_C\preceq\Phi_{CX}$
\item[(5)] $\Phi_X\preceq\Phi_{CX}$
\end{itemize}

\end{lemma}

\begin{proof}
For (1), note that any simple matroid $M$  of rank $r$ virtually conforming to $\Phi_{Y_1}$ is a restriction of the vector matroid of a matrix $A$ of the following form:

\begin{center}
\begin{tabular}{ |c|ccc|c|c|c|}
\hline
\multirow{2}{*}{0}&1&0&1&$1\cdots1$&$0\cdots0$&$1\cdots1$\\
&0&1&1&$0\cdots0$&$1\cdots1$&$1\cdots1$\\
\hline
$\Gamma$-frame matrix&\multicolumn{3}{c|}{0}&$I$&$I$&$I$\\
\hline
\end{tabular}
\end{center}

If we label the sets of rows and columns of $A$ as $B$ and $E$ respectively, and the first row as $x$, then we see that $A[B-x,E]$ is a $\Gamma$-frame matrix. If we let $X=\{x\}$, then we see that $M$ conforms to $\Phi_X$.

For (2), consider the matrix $A$ above. Note that it is obtained by contracting $c$ in the following matrix:

\begin{center}
\begin{tabular}{ |c|ccc|c|c|c|c|}
\multicolumn{7}{c}{}&\multicolumn{1}{c}{$c$}\\
\hline
\multirow{3}{*}{0}&0&0&1&0$\cdots$0&0$\cdots$0&1$\cdots1$&1\\
&1&0&0&$1\cdots1$&$0\cdots0$&$0\cdots0$&1\\
&0&1&0&$0\cdots0$&$1\cdots1$&$0\cdots0$&1\\
\hline
$\Gamma$-frame matrix&\multicolumn{3}{c|}{0}&$I$&$I$&$I$&0\\
\hline
\end{tabular}
\end{center}
Removing $c$ from this matrix, we obtain a $\Gamma$-frame matrix. Therefore, $M$ conforms to $\Phi_C$.

For (3), any matroid $M$ conforming to $\Phi_{Y_0}$ is the vector matroid of a matrix of the following form, where $v$ is an arbitrary column vector:

\begin{center}
\begin{tabular}{ |c|c|}
\hline
&\\
$\Gamma$-frame matrix&$v$\\
&\\
\hline
\end{tabular}
\end{center}

Let $A$ be the matrix below. Label its sets of rows and columns as $B$ and $E$ respectively, and let $c$ be the last column, with $C=\{c\}$.

\begin{center}
\begin{tabular}{|c|c|c|}
\hline
0&1&1\\
\hline
&&\\
$\Gamma$-frame matrix&0&$v$\\
&&\\
\hline
\end{tabular}
\end{center}

Note that $M$ is isomorphic to $M(A)/C$. Since $A[B,E-C]$ is a $\Gamma$-frame matrix, we see that $M$ conforms to $\Phi_C$.

For (4), let $A$ be a matrix conforming to $\Phi_C$ and let $M=M(A)/C$ be the corresponding matroid conforming to $\Phi_C$. If the column of $A$ indexed by $C$ is a zero column, then construct the matrix $\bar{A}$ by adding a unit row, indexed by $X$, whose nonzero entry is in the column indexed by $C$. One readily sees that $\bar{A}$ conforms to $\Phi_{CX}$ and that the corresponding matroid $M(\bar{A})/C$ is equal to $M$. Otherwise, if the column of $A$ indexed by $C$ has a nonzero entry, then one readily sees that $A$ conforms to $\Phi_{CX}$ by considering the row containing the nonzero entry to be indexed by $X$.

For (5), any matroid $M$ conforming to $\Phi_D$ is the vector matroid of a matrix of the following form, where $v$ is an arbitrary row vector:
\begin{center}
\begin{tabular}{|c|}
\hline
$v$\\
\hline
\\
$\Gamma$-frame matrix\\
\\
\hline
\end{tabular}
\end{center}

Consider the following matrix $A$, whose last column is indexed by $\{c\}=C$:
\begin{center}
\begin{tabular}{|c|c|}
\hline
$v$&1\\
\hline
$0$&1\\
\hline
&\\
$\Gamma$-frame matrix&0\\
&\\
\hline
\end{tabular}
\end{center}
The matroid $M$ is isomorphic to $M(A)/c$, which conforms to $\Phi_{CX}$.
\end{proof}

\begin{lemma} \label{yshift}
Let $\Phi$ be a template with $y\in Y_1$. Let $\Phi'$ be the template obtained from $\Phi$ by removing $y$ from $Y_1$ and placing it in $Y_0$. Then $\Phi'\preceq \Phi$.
\end{lemma}

\begin{proof}
Any matrix respecting $\Phi'$ virtually respects $\Phi$ by adding column $y$ only to the zero $Z$ column. Thus, any matroid conforming to $\Phi'$ weakly conforms to $\Phi$.
\end{proof}

We call the operation described in Lemma ~\ref{yshift} a \textit{$y$-shift}.

\begin{definition}
 Let $\Phi=(\Gamma,C,X,Y_0,Y_1,A_1,\Delta,\Lambda)$ be a frame template over a finite field $\mathbb{F}$. We say that $\Phi$ is in \textit{standard form} if there are disjoint sets $C_0,C_1,X_0,$ and $X_1$ such that $C=C_0\cup C_1$, such that $X=X_0\cup X_1$, such that $A_1[X_0,C_0]$ is an identity matrix, and such that $A_1[X_1,C]$ is a zero matrix.
\end{definition}

 Figure ~\ref{fig:A' standard}, with the stars representing arbitrary matrices, shows a matrix that virtually respects a template in standard form. Note that if $\Phi$ is in standard form, $|C_0|=|X_0|$. Also note that any of $C_0,C_1,X_0,$ or $X_1$ may be empty. Finally, note that we have defined standard form for frame templates over any finite field, not just binary frame templates.

\begin{figure}
\begin{center}
\begin{tabular}{ r|c|c|cccc| }
\multicolumn{2}{c}{}&\multicolumn{1}{c}{$Z$}&\multicolumn{1}{c}{$Y_0$}&\multicolumn{1}{c}{$Y_1$}&\multicolumn{1}{c}{$C_0$}&\multicolumn{1}{c}{$C_1$}\\
\cline{2-7}
$X_0$&columns from $\Lambda|X_0$&0&\multicolumn{2}{c|}{\multirow{2}{*}{$*$}}&\multicolumn{1}{c|}{$I$}&$*$\\
\cline{2-3} \cline{6-7}
$X_1$&columns from $\Lambda|X_1$&0&&&\multicolumn{2}{|c|}{0}\\
\cline{2-7}
&\multirow{5}{*}{$\Gamma$-frame matrix}&\multirow{5}{*}{unit or zero columns}&\multicolumn{4}{c|}{\multirow{5}{4em}{rows from  $\Delta$}}\\
&&&&&&\\
&&&&&&\\
&&&&&&\\
&&&&&&\\
\cline{2-7}
\end{tabular}
\end{center}
\caption{Standard form}
  \label{fig:A' standard}
\end{figure}

\begin{lemma}
\label{standard}
 Every frame template $\Phi=(\Gamma,C,X,Y_0,Y_1,A_1,\Delta,\Lambda)$ is equivalent to a frame template in standard form.
\end{lemma}

\begin{proof}
 Choose a basis $C_0$ for $M(A_1[X,C])$, and let $C_1=C-C_0$. Repeatedly perform operation (5) to obtain a template $\Phi'$ where $A_1[X,C_0]$ consists of an identity matrix on top of a zero matrix. Each use of operation (5) results in an equivalent template; therefore, $\Phi\sim\Phi'$. Let $X_0\subseteq X$ index the rows of the identity matrix, and let $X_1\subseteq X$ index the rows of the zero matrix. Since $C_0$ is a basis for $M(A_1[X,C])$, the matrix $A_1[X,C_1]$ must be a zero matrix as well. Thus, $\Phi'$ is in standard form.
\end{proof}

Throughout the rest of this paper, we will implicitly use Lemma ~\ref{standard} to assume that all templates are in standard form. Also, the operations (1)-(12) to which we will refer throughout the rest of this paper are the operations (1)-(8) from Proposition ~\ref{reductions} and (9)-(12) from Definition ~\ref{weaklyconforming}.

\begin{lemma}
 \label{PhiD}
If $\Phi=(\{1\},C,X,Y_0,Y_1,A_1,\Delta,\Lambda)$ is a binary frame template with $\Lambda|X_1$ nontrivial, then $\Phi_X\preceq\Phi$.
\end{lemma}

\begin{proof}
Perform operations (2) and (3) on $\Phi$ to obtain the following template, where $\lambda$ is an element of $\Lambda$ with $\lambda_x\neq0$ for some $x\in X_1$:
\[(\{1\},C,X,Y_0,Y_1,A_1,\{0\},\{\mathbf{0}, \lambda\}).\] On this template, repeatedly perform operation (7), then (8), then (4), and then (10) until the following template is obtained: \[(\{1\},\emptyset,X_1,\emptyset,\emptyset,[\emptyset],\{0\},\{\mathbf{0}, \lambda\}).\] On this template, repeatedly perform operation (5) to obtain a template that is identical to the previous one except that the support of $\lambda$ contains only one element of $X_1$. On this template, repeatedly perform operation (6) to obtain the following template, where $x\in X_1$: \[(\{1\},\emptyset,\{x\},\emptyset,\emptyset,[\emptyset],\{0\},\mathbb{Z}/2\mathbb{Z}).\] This template is $\Phi_X$.
\end{proof}

\begin{lemma}
\label{PhiC}
 If $\Phi=(\{1\},C,X,Y_0,Y_1,A_1,\Delta,\Lambda)$ is a binary frame template, then either $\Phi_C\preceq\Phi$ or $\Phi$ is equivalent to a template with $C_1=\emptyset$.
\end{lemma}

\begin{proof}
Suppose there is an element $\delta\in\Delta|C$ that is not in the row space of $A_1[X,C]$. Repeatedly perform operations (4) and (10) on $\Phi$ until the following template is obtained:
\[(\{1\},C,X,\emptyset,\emptyset,A_1[X,C],\Delta|C,\Lambda).\]
On this template, perform operations (2) and (3) to obtain the following template:
\[(\{1\},C,X,\emptyset,\emptyset,A_1[X,C],\{\mathbf{0}, \delta\},\{0\}).\]
Every matrix virtually respecting this template is row equivalent to a matrix virtually respecting a template that is identical to the previous template except that there is the additional condition that $\delta|C_0$ is a zero vector. Note that $\delta|C_1$ is nonzero since, in the previous template, $\delta$ was not in the row space of $A_1[X,C]$.
Now, on the current template, repeatedly perform operation (7) and then operation (6) to obtain the following template:
\[\Phi'=(\{1\},C_1,\emptyset,\emptyset,\emptyset,[\emptyset],\{\mathbf{0}, \delta|C_1\},\{0\}).\]

Now, any matroid $M$ conforming to $\Phi'$ is obtained by contracting $C_1$ from $M(A)$, where $A$ is a matrix conforming to $\Phi'$. By contracting any single element $c\in C_1$, where $\delta_c=1$, we turn the rest of the elements of $C_1$ into loops. So $C_1-c$ is deleted to obtain $M$. Thus, $M$ conforms to the template
\[(\{1\},\{c\},\emptyset,\emptyset,\emptyset,[\emptyset],\mathbb{Z}/2\mathbb{Z},\{0\}),\]
which is $\Phi_C$. Similarly, the converse is true that any matroid conforming to $\Phi_C$ conforms to $\Phi'$. Thus, $\Phi_C\sim\Phi'\preceq\Phi$.

Now suppose that every element of $\Delta|C$ is in the row space of $A_1[X,C]$. Thus, contraction of $C_0$ turns the elements of $C_1$ into loops, and contraction of $C_1$ is the same as deletion of $C_1$. By deleting $C_1$ from every matrix virtually conforming to $\Phi$, we see that $\Phi$ is equivalent to a template with $C_1=\emptyset$.
\end{proof}

\begin{lemma}
\label{PhiCD}
 If $\Phi=(\{1\},C,X,Y_0,Y_1,A_1,\Delta,\Lambda)$ is a binary frame template, then one of the following is true:
\begin{itemize}
\item $\Phi_C\preceq\Phi$
\item $\Phi$ is equivalent to a template with $\Lambda|X_1$ nontrivial and $\Phi_X\preceq\Phi$
\item $\Phi$ is equivalent to a template with $\Lambda|X_0$ nontrivial and $\Phi_{CX}\preceq\Phi$
\item $\Phi$ is equivalent to a template with $\Lambda$ trivial and $C=\emptyset$.
\end{itemize}

\end{lemma}

\begin{proof}
By Lemmas ~\ref{PhiD} and ~\ref{PhiC}, we may assume that $\Lambda|X_1$ is trivial and that $C_1=\emptyset$.

First, suppose there exist elements $\delta\in\Delta|C_0$ and $\lambda\in\Lambda|X_0$ such that there are an odd number of natural numbers $i$ with $\delta_i=\lambda_i=1$. Thus, $\Lambda|X_0$ is nontrivial. Repeatedly perform operations (4) and (10) on $\Phi$ until the following template is obtained:
\[(\{1\},C_0,X,\emptyset,\emptyset,A_1[X,C_0],\Delta|C_0,\Lambda).\]
On this template, repeatedly perform operation (6) to obtain the following template:
\[\Phi'=(\{1\},C_0,X_0,\emptyset,\emptyset,A_1[X_0,C_0],\Delta|C_0,\Lambda|X_0).\]
 Perform operations (2) and (3) on $\Phi'$ to obtain the following template:
\[(\{1\},C_0,X_0,\emptyset,\emptyset,A_1[X_0,C_0],\{\mathbf{0}, \delta\},\{\mathbf{0}, \lambda\}).\]
Any matroid conforming to this template is obtained by contracting $C_0$. If $\delta$ is in the row labeled by $r$ and $\lambda$ is in the column labeled by $c$, then when $C_0$ is contracted, 1 is added to the entry of the $\Gamma$-frame matrix in row $r$ and column $c$. Otherwise, the entry remains unchanged when $C$ is contracted. We see then that this template is equivalent to $\Phi_{CX}$, where 1s are used to replace $\delta$ and $\lambda$.

Thus, we may assume that for every element $\delta\in\Delta|C_0$ and $\lambda\in\Lambda|X_0$, there are an even number of natural numbers $i$ such that $\delta_i=\lambda_i=1$. This implies that contraction of $C$ has no effect on the $\Gamma$-frame matrix. So $\Phi$ is equivalent to a template with $\Lambda|X_0$ trivial. Therefore, since $\Lambda|X_1$ is trivial, we see that $\Lambda$ is trivial. Note that operation (7) is a reduction that produces an equivalent template, since $C$ must be contracted to produce a matroid that conforms to a template. By repeatedly performing operation (7), we obtain a template equivalent to $\Phi$ with $C=\emptyset$.
\end{proof}

\begin{lemma}
 \label{PhiY0}
If $\Phi=(\{1\},C,X,Y_0,Y_1,A_1,\Delta,\Lambda)$ is a binary frame template with $\Lambda$ trivial and with $C=\emptyset$, then either $\Phi_{Y_0}\preceq\Phi$ or $\Phi$ is equivalent to a template with $\Delta$ trivial.
\end{lemma}

\begin{proof}
First, suppose there is an element $\delta\in\Delta$ that is not in the row space of $A_1=A_1[X_1,(Y_0\cup Y_1)]$. Recall that a $y$-shift is the operation described in Lemma ~\ref{yshift}. Repeatedly perform $y$-shifts to obtain the following template, where $Y'_0=Y_0\cup Y_1$:
\[(\{1\},\emptyset,X,Y'_0,\emptyset,A_1,\Delta,\{0\}).\]
On this template, perform operation (3) to obtain the following template:
\[(\{1\},\emptyset,X,Y'_0,\emptyset,A_1,\{\mathbf{0}, \delta\},\{0\}).\]

Choose a basis $B'$ for $M(A_1)$. By performing elementary row operations on every matrix virtually respecting $\Phi$, we may assume that $A_1[X,B']$ consists of an identity matrix with zero rows below it and that $\delta|B'$ is the zero vector. By assumption, there is some element $y\in (Y'_0-B')$ such that $\delta_y$ is nonzero. Thus, we can repeatedly perform operation (10) to obtain the following template:
\[(\{1\},\emptyset,X,B'\cup y,\emptyset,A_1[X,B'\cup y],\{\mathbf{0}, \delta|(B'\cup y)\},\{0\}).\]
Now, we can repeatedly perform operation (6) and then operation (12) to obtain the following template:
\[(\{1\},\emptyset,\emptyset,\{y\},\emptyset,[\emptyset],\mathbb{Z}/2\mathbb{Z},\{0\}),\]
which is $\Phi_{Y_0}$.

Now suppose that every element $\delta\in\Delta$ is in the row space of $A_1=A_1[X,(Y_0\cup Y_1)]$. Since $\Lambda$ is trivial, by performing elementary row operations on every matrix virtually respecting $\Phi$, we obtain a template equivalent to $\Phi$ with $\Delta$ trivial.
\end{proof}

\begin{lemma}
\label{PhiY1}
 Let $\Phi=(\{1\},C,X,Y_0,Y_1,A_1,\Delta,\Lambda)$ be a binary frame template with $\Lambda$ and $\Delta$ trivial. If $M(A_1[X_1,(Y_0\cup Y_1)])$ has a circuit $Y'$ with $|Y'\cap Y_1|\geq3$, then $\Phi_{Y_1}\preceq\Phi$.
\end{lemma}

\begin{proof}
 Any matroid conforming to $\Phi$ is obtained by contracting $C$. Since $\Lambda$ and $\Delta$ are trivial, we may assume that $C=X_0=\emptyset$ and therefore that $X=X_1$. Repeatedly perform operation (4) and then operation (10) on $\Phi$ to obtain the following template:
\[(\{1\},\emptyset,X,Y_0\cap Y',Y_1\cap Y',A_1[X,Y'],\{0\},\{0\}).\]
Choose any 3-element subset of $Y'\cap Y_1$ and call it $Y''$. Repeatedly perform $y$-shifts to obtain the following template:
\[(\{1\},\emptyset,X,Y'-Y'',Y'',A_1[X,Y'],\{0\},\{0\}).\]
On this template, repeatedly perform operation (11) to obtain the following template:
\[(\{1\},\emptyset,X',\emptyset,Y'',A_1[X',Y''],\{0\},\{0\}),\]
where $X'$ is the subset of $X$ that remains after $Y'-Y''$ is contracted. On this template, repeatedly perform operations (5) and (6) to obtain the following template, where $X''$ is a 2-element subset of $X'$:
\[(\{1\},\emptyset,X'',\emptyset,Y'',\begin{bmatrix}
1& 0 &1\\
0& 1 & 1
\end{bmatrix},\{0\},\{0\}).\]
This template is $\Phi_{Y_1}$.
\end{proof}

\begin{lemma}
 \label{simpleY1}
If $\Phi$ is a frame template with $\Delta$ trivial, then $\Phi$ is equivalent to a template $\Phi'$ where $A_1[X,Y_1]$ is a matrix with every column nonzero and where no column is a copy of another. Moreover, if $\Phi$ is a binary frame template, then $M(A_1[X,Y_1])$ is simple.
\end{lemma}

\begin{proof}
 Let $A$ be a matrix that virtually conforms to $\Phi$. Since $\Delta$ is trivial, the columns of $A$ indexed by elements of $Z$ are formed by placing a column of $A_1[X,Y_1]$ on top of a unit column or a zero column. These columns can be made using any copy of the same column of $A_1[X,Y_1]$; so only one copy is needed. If any column of $A_1[X,Y_1]$ is a zero column, then any column indexed by an element of $Z$ that is made with this zero column can also be made as a column indexed by an element of $E-(Z\cup Y_0\cup Y_1\cup C)$ and choosing for the element of $\Lambda$ the zero vector. Thus, no zero columns of $A_1[X,Y_1]$ are needed.

In the binary case, $M(A_1[X,Y_1])$ has no parallel elements because any such elements index copies of the same column. Also, $M(A_1[X,Y_1])$ has no loops because every column of $A_1[X,Y_1]$ is nonzero. Therefore, $M(A_1[X,Y_1])$ is simple.
\end{proof}

\begin{lemma}
 \label{HL}
Let $\Phi$ be a binary frame template. Then at least one of the following is true:

\begin{itemize}
\item[(i)]$\Phi_0\sim\Phi$
\item[(ii)]$\Phi'\preceq \Phi$ for some $\Phi'\in\{\Phi_X,\Phi_C,\Phi_{CX},\Phi_{Y_0},\Phi_{Y_1}\}$
\item[(iii)]$\Phi$ is equivalent to a template where $C=\emptyset$, where $\Lambda$ and $\Delta$ are trivial, and where $A_1$ is of the following form, with $Y_0=V_0\cup V_1$, with $L$ an arbitrary binary matrix, and with each column of $H$ containing at most two nonzero entries:
\begin{center}
\begin{tabular}{ |c|c|c| }
\multicolumn{1}{c}{$Y_1$}&\multicolumn{1}{c}{$V_0$}&\multicolumn{1}{c}{$V_1$}\\
\hline
$I$&0&$H$\\
\hline
0&$I$&$L$\\
\hline
\end{tabular}
.
\end{center}
\end{itemize}
\end{lemma}

\begin{proof}
 Suppose neither (i) nor (ii) holds. By Lemma ~\ref{PhiCD}, we may assume that $\Lambda$ is trivial and $C=\emptyset$. By Lemma ~\ref{PhiY0}, we may assume that $\Delta$ is trivial. By Lemma ~\ref{PhiY1}, every dependent set of $M(A)=M(A_1[X_1,(Y_0\cup Y_1)])$ has an intersection with $Y_1$ with size at most 2. So by elementary row operations, we may assume that $A_1$ is of the following form, where  $Y_0=V_0\cup V_1$, where $L$ is an arbitrary binary matrix, where $K$ consists of unit and zero columns, and where each column of $H$ contains at most two nonzero entries:

\begin{center}
\begin{tabular}{ |c|c|c|c| }
\multicolumn{2}{c}{$Y_1$}&\multicolumn{1}{c}{$V_0$}&\multicolumn{1}{c}{$V_1$}\\
\hline
$I$&$K$&0&$H$\\
\hline
0&0&$I$&$L$\\
\hline
\end{tabular}
.
\end{center}
However, by Lemma ~\ref{simpleY1}, we may assume that $K$ is an empty matrix. Thus, (iii) holds.
\end{proof}

\begin{theorem} \label{minimal}
Let $\Phi$ be a binary frame template. Then at least one of the following is true:

\begin{itemize}
\item[(i)]$\Phi_0\sim\Phi$
\item[(ii)]$\Phi'\preceq \Phi$ for some $\Phi'\in\{\Phi_X,\Phi_C,\Phi_{CX},\Phi_{Y_0},\Phi_{Y_1}\}$
\item[(iii)]There exist $k,l\in\mathbb{Z}_+$ such that no simple, vertically $k$-connected matroid with at least $l$ elements either virtually conforms or virtually coconforms to $\Phi$.
\end{itemize}
\end{theorem}

\begin{proof} 

Suppose for contradiction that none of outcomes (i)-(iii) hold for $\Phi$. By Lemma ~\ref{HL}, outcome (iii) of that lemma holds. Note that any simple matroid $N$ virtually conforming to $\Phi$ is  a restriction of a matroid $M$ represented by a matrix of the following form, where $Z=Z_0\cup Z_1$, where $Y_0=V_0\cup V_1$, and where the $\Gamma$-frame matrix has $n$ rows and has a vector matroid isomorphic to the cycle matroid of the graph $K_{n+1}$:

\begin{center}
\begin{tabular}{ r|c|cccc|c|c|c| }
\multicolumn{2}{c}{}&\multicolumn{4}{c}{$Z_0$}&\multicolumn{1}{c}{$Z_1$}&\multicolumn{1}{c}{$V_0$}&\multicolumn{1}{c}{$V_1$}\\
\cline{2-9}
\multirow{5}{*}{$X$}&\multirow{5}{*}{0}&$1\cdots1$&&&&\multirow{4}{*}{$I$}&\multirow{4}{*}{0}&\multirow{4}{*}{$H$}\\
&&&$1\cdots1$&&&&&\\
&&&&$\cdots$&&&&\\
&&&&&$1\cdots1$&&&\\
\cline{3-9}
&&\multicolumn{4}{c|}{0}&0&$I$&$L$\\
\cline{2-9}
&$\Gamma$-frame matrix&$I$&$I$&$\cdots$&$I$&0&0&0\\
\cline{2-9}
\end{tabular}
\end{center}
Also recall from the definition of conforming to a template that $Y_0\subseteq E(N)$.

We see that
\begin{align*}
\lambda_N(Y_0\cup (Z_1\cap E(N)))
&\leq \lambda_M(Y_0\cup Z_1)\\
&= r_M(Y_0\cup Z_1)+r_M(E-(Y_0\cup Z_1))-r(M)\\
&= |V_0|+|Y_1| + |Y_1|+n-(|Y_1|+|V_0|+n)\\
&= |Y_1|.
\end{align*}

Note that each column of the above matrix, except possibly those columns indexed by $V_1$, has at most two nonzero entries. Thus, $M$ is graphic and $\Phi$ is trivial if $V_0=\emptyset$. Since (i) does not hold, $\Phi$ is nontrivial. Therefore, $V_0\neq\emptyset$, and $E(N)-(Y_0\cup Z_1)$ is not spanning. Thus, if $k>|Y_1|+1$, then $N$ is not vertically $k$-connected unless $Y_0\cup(Z_1\cap E(N)) $ is spanning in $N$.  This implies that $n=0$; in that case, $N$ is only simple if the $\Gamma$-frame matrix is a $0\times0$ matrix. This implies that $|E(N)|\leq |Y_0\cup Y_1|$. So if $l>|Y_0\cup Y_1|$, then no simple, vertically $k$-connected matroid with at least $l$ elements virtually conforms to $\Phi$.

Now, consider a simple matroid $N^*$ which virtually coconforms to $\Phi$. Then $N$ is a restriction of $M$ with $Y_0\subseteq E(N)$. Since a matroid and its dual have the same connectivity function, we have $\lambda_{N^*}(Y_0\cup (Z_1\cap E(N))\leq |Y_1|$. So if $k>|Y_1|+1$, then $N^*$ is not vertically $k$-connected unless either $Y_0\cup (Z_1\cap E(N))$ or $E(N)-(Z_1\cup Y_0)$ is spanning in $N^*$, implying that either $E(N)-(Z_1\cup Y_0)$ or $Y_0\cup (Z_1\cap E(N))$ is independent in $N$. If $E(N)-(Z_1\cup Y_0)$ is independent in $N$, then
\begin{align*}
|E(N)-(Z_1\cup Y_0)|&=r_N(E(N)-(Z_1\cup Y_0))\\
&\leq r_M(E(M)-(Z_1\cup Y_0))\\
&=|Y_1|+n.
\end{align*}
By the formula for corank, we have
\begin{align*}
r_{N^*}(E(N)-(Z_1\cup Y_0))&\leq r_{M^*}(E(N)-(Z_1\cup Y_0))\\
&= |E(N)-(Z_1\cup Y_0)|+r_M(Z_1\cup Y_0)-r(M)\\
&\leq |Y_1|+n+|Y_1|+|V_0|-(|Y_1|+|V_0|+n)\\
&=|Y_1|.
\end{align*}
Since $N^*$ is simple and binary, we have $|E(N)-(Z_1\cup Y_0)|\leq 2^{|Y_1|}-1$. This implies that $|E(N)|\leq 2^{|Y_1|}-1+|Y_1|+|Y_0|$. Thus, if we set $l$ greater than this value, then no simple, vertically $k$-connected matroid with at least $l$ elements virtually coconforms to $\Phi$ unless $Y_0\cup (Z_1\cap E(N))$ is independent in $N$. Since (iii) does not hold, this must be true for some matroid $N$. In particular, $Y_0=V_0\cup V_1$ is independent in $N$, implying that $H$ is a linearly independent matrix. 

Let $P$ denote the matrix 
\[P=
\left[
\begin{array}{cc}
1& 0\\
0& 1\\
0&1\\
\hline
1&1\\
\end{array}
\right].
\]
 Suppose $A_1[X,V_1]$ has $P$ as a submatrix, with the first three rows of $P$ contained in $H$ and the last row of $P$ contained in $L$. Then $A_1$ contains the following submatrix, with the first three columns contained in $A_1[X,Y_1]$ and the last two contained in $A_1[X,V_1]$:

\[
\left[
\begin{array}{ccc|cc}
1&0&0&1&0\\
0&1&0&0&1\\
0&0&1&0&1\\
\hline
0&0&0&1&1\\
\end{array}
\right]. 
\]
After contracting all other elements of $Y_1$ by repeatedly performing $y$-shifts and operation (12), the columns of this submatrix form a circuit in $M(A_1)$ whose intersection with $Y_1$ has size 3. However, we have already deduced by Lemma ~\ref{PhiY1} that this is impossible. Therefore, $A_1$ does not contain $P$ as a submatrix, with the first three rows of $P$ contained in $H$ and the last row of $P$ contained in $L$. We will refer to this fact by saying that $A_1$ has no \textit{$P$-configuration}.

Let $\{1,2,\dots, m\}$ be the rows of $L$. (So $|V_0|=m$.) Let $S_i$ be the submatrix of $H$ obtained by restricting $H$ to the columns $j$ such that $L_{i,j}=1$. Recall that $H$, and therefore $S_i$, contain at most two nonzero entries per column. Also, since $H$ is linearly independent, each column has at least one nonzero entry, and no column is a copy of another. Suppose a column $e$ of $S_i$ contains exactly two nonzero entries. Since $A_1$ has no $P$-configuration, all other columns of $S_i$ must contain a nonzero entry in exactly one of the same rows as $e$. Suppose that there are columns $f$ and $g$ in $S_i$ such that $f$ contains a nonzero entry in one of the same rows as $e$, but $g$ contains a nonzero entry in the other row. Then $S_i$ contains the following submatrix:
\[\begin{blockarray}{ccc}
e & f & g \\
\begin{block}{[ccc]}
1 & 1 & 0 \\
1 & 0 & 1 \\
\end{block}
\end{blockarray}.\]
Since $H$ is a linearly independent matrix, $f$ or $g$ (say $f$) must have an additional nonzero entry in $H$. To avoid $f$ and $g$ forming a $P$-configuration, $g$ must have an additional nonzero entry in the same row as $f$. Therefore, $S_i$ contains the following submatrix:
\[\begin{blockarray}{ccc}
e & f & g \\
\begin{block}{[ccc]}
1 & 1 & 0 \\
1 & 0 & 1 \\
0 & 1 & 1\\
\end{block}
\end{blockarray}.\]
Since each column of $H$ contains at most two nonzero entries, $\{e,f,g\}$ is a dependent set of columns, contradicting the assumption that $H$ is linearly independent.

Therefore, we deduce that each $S_i$ either consists entirely of unit columns or contains a row $s_i$ consisting entirely of 1s. Note that each $S_i$ is the incidence matrix of a star, with possibly one row removed. We will call $s_i$ the \textit{star center} of row $i$. If $S_i$ consists entirely of unit columns, then we define its star center to be $s_i=\emptyset$.

If the sets of columns of all the $S_i$ are pairwise disjoint, then by adding each row $i$ to its star center $s_i$, we see that every matroid virtually conforming to $\Phi$ can be represented by a matrix with at most two nonzero entries per column. Thus, $\Phi$ is trivial, contradicting the assumption that (i) does not hold. Also, if $i$ and $j$ are distinct rows of $L$ with distinct star centers $s_i$ and $s_j$, then $S_i$ and $S_j$ can have at most one column in common because otherwise, the columns they have in common form a linearly dependent set in $H$.

Now suppose there are $S_i$ and $S_j$ with $s_i=s_j$. Also, suppose that neither $S_i$ nor $S_j$ is a submatrix of the other. Then $A_1$ contains the following submatrix. In fact, after repeatedly performing $y$-shifts, operation  (11), and operation (10), we may assume that $A_1$ is the following matrix, with the first three columns indexed by $Y_1$, the next two indexed by $V_0$, and the last three by $V_1$:
\[
\left[
\begin{array}{ccc|cc|ccc}
1&0&0&0&0&1&1&1\\
0&1&0&0&0&0&1&0\\
0&0&1&0&0&0&0&1\\
\hline
0&0&0&1&0&1&1&0\\
0&0&0&0&1&0&1&1\\
\end{array}
\right].
\]
Add the fourth row to the first, and swap the fourth and sixth columns to obtain the following matrix:
\[
\left[
\begin{array}{ccc|cc|ccc}
1&0&0&0&0&1&0&1\\
0&1&0&0&0&0&1&0\\
0&0&1&0&0&0&0&1\\
\hline
0&0&0&1&0&1&1&0\\
0&0&0&0&1&0&1&1\\
\end{array}
\right].
\]
The last two columns of this matrix contain a $P$-configuration.

Now suppose there are matrices $S_i$ and $S_j$ so that $S_j$ is a submatrix of $S_i$. Then $A_1$ contains a submatrix obtained by deleting columns from a matrix of the following form, where the left portion comes from the set $V_0$, the upper-right portion comes from the matrix $H$, the lower-left portion comes from the matrix $L$, and $x$ is 1 or 0 depending on whether or not the last column is contained in $S_j$:
\[
\left[
\begin{array}{cc|ccccccc}
0&0&1&\cdots&1&1&\cdots&1&1\\
0&0&1&&&&&&0\\
\vdots&\vdots&&\ddots&&&&&\vdots\\
0&0&&&1&&&&0\\
0&0&&&&1&&&0\\
\vdots&\vdots&&&&&\ddots&&\vdots\\
0&0&&&&&&1&0\\
\hline
1&0&1&\cdots&1&1&\cdots&1&1\\
0&1&0&\cdots&0&1&\cdots&1&x
\end{array}
\right].
\]

Choose any column contained in $S_j$ and perform row operations so that this column becomes a unit column with nonzero entry in $L$. Then we obtain the following matrix: 
\[
\left[
\begin{array}{cc|cccccccc}
0&1&1&\cdots&1&0&0&\cdots&0&x+1\\
0&0&1&&&&&&&0\\
\vdots&\vdots&&\ddots&&&&&&\vdots\\
0&0&&&1&&&&&0\\
0&1&&&&0&1&\cdots&1&x\\
0&0&&&&&1&&&0\\
\vdots&\vdots&&&&&&\ddots&&\vdots\\
0&0&&&&&&&1&0\\
\hline
1&1&1&\cdots&1&0&\cdots&\cdots&0&x+1\\
0&1&0&\cdots&0&1&\cdots&\cdots&1&x
\end{array}
\right].
\]

Now, by swapping the appropriate columns, we obtain the following:
\[
\left[
\begin{array}{cc|cccccccc}
0&0&1&\cdots&1&1&0&\cdots&0&x+1\\
0&0&1&&&&&&&0\\
\vdots&\vdots&&\ddots&&&&&&\vdots\\
0&0&&&1&&&&&0\\
0&0&&&&1&1&\cdots&1&x\\
0&0&&&&&1&&&0\\
\vdots&\vdots&&&&&&\ddots&&\vdots\\
0&0&&&&&&&1&0\\
\hline
1&0&1&\cdots&1&1&0&\cdots&0&x+1\\
0&1&0&\cdots&0&1&\cdots&\cdots&1&x
\end{array}
\right].
\]

We see that in this new matrix, $S_i$ and $S_j$ have only one column in common and $s_i\neq s_j$. The last column is in $S_i$ if $x=0$ and $S_j$ if $x=1$. Thus, this case reduces to the final case that remains to be checked: for all $i$ and $j$, we have $s_i\neq s_j$ and $S_i$ and $S_j$ have at most one column in common. Since each column of $H$ contains at most two nonzero entries, and since all $S_i$ have distinct star centers, we see that a column of $H$ can be contained in at most two $S_i$. By adding each row $i$ to its star center $s_i$, one can see that every matrix virtually conforming to $\Phi$ can be rewritten so that every column contains at most two nonzero entries. Therefore, $\Phi$ is trivial, and (i) holds.

This completes the contradiction and proves the result.
\end{proof}

Outcome (iii) of Theorem ~\ref{minimal} only occurs in very specific situations. In fact, due to connectivity considerations, it is not needed in order to use Corollary ~\ref{weakframe}.

\begin{definition}
\label{describes}
Let $\mathcal{M}$ be a minor-closed class of binary matroids, and suppose there exist $k,l,m\in \mathbb{Z}_+$ and a set $\mathcal{T}_{\mathcal{M}}=\{\Phi_1,\dots, \Phi_s, \Psi_1,\dots, \Psi_t\}$ of binary frame templates such that
\begin{itemize}
\item $\mathcal{M}$ contains each of the classes $\mathcal{M}_w(\Phi_1),\dots,\mathcal{M}_w(\Phi_s)$,
\item $\mathcal{M}$ contains the duals of the matroids in each of the classes $\mathcal{M}_w(\Psi_1)$,$\dots$,$\mathcal{M}_w(\Psi_t)$,
\item if $M$ is a simple vertically $k$-connected member of $\mathcal{M}$ with at least $l$ elements and with no $PG(m-1,2)$ minor, then either $M$ is a member of at least one of the classes $\mathcal{M}_v(\Phi_1),\dots,\mathcal{M}_v(\Phi_s)$ or $M^*$ is a member of at least one of the classes $\mathcal{M}_v(\Psi_1),\dots,\mathcal{M}_v(\Psi_t)$, and
\item for each template $\Phi\in\mathcal{T}_{\mathcal{M}}$, either $\Phi$ is trivial or $\Phi'\preceq \Phi$ for some $\Phi'\in\{\Phi_X,\Phi_C,\Phi_{CX},\Phi_{Y_0},\Phi_{Y_1}\}$.
\end{itemize}

We say that $\mathcal{T}_{\mathcal{M}}$ \textit{describes} $\mathcal{M}$.
\end{definition}

By combining Corollary ~\ref{weakframe} with Theorem ~\ref{minimal}, one can observe that every proper minor-closed class $\mathcal{M}$ of binary matroids can be described by a set of templates. Moreover, that set is nonempty if and only if $\mathcal{M}$ contains all graphic matroids or all cographic matroids.

\begin{corollary}
\label{Y0Y1}
Let $\mathcal{M}$ be a minor-closed class of binary matroids, and let $\{\Phi_1,\dots, \Phi_s, \Psi_1,\dots, \Psi_t\}$ be a set of templates describing $\mathcal{M}$. If any of these templates is nontrivial, then $\mathcal{M}$ contains $\mathcal{M}(\Phi_{Y_0})$, $\mathcal{M}(\Phi_{Y_1})$, $\mathcal{M}^*(\Phi_{Y_0})$, or $\mathcal{M}^*(\Phi_{Y_1})$.
\end{corollary}

\begin{proof}
Let $\Phi$ be a nontrivial template in the set $\{\Phi_1,\dots, \Phi_s\}$. By Definition ~\ref{describes} and Lemma ~\ref{YCD}, either $\Phi_{Y_0}\preceq\Phi$ or $\Phi_{Y_1}\preceq\Phi$. If $\Phi_{Y_0}\preceq\Phi$, then 
\[\mathcal{M}(\Phi_{Y_0})\subseteq\mathcal{M}_v(\Phi_{Y_0})\subseteq\mathcal{M}_v(\Phi)\subseteq\mathcal{M},\]
where the first containment holds because every matroid conforming to a template also virtually conforms to it, the second containment holds by definition of $\preceq$, and the third containment holds by Definition ~\ref{describes}. In the case where $\Phi_{Y_1}\preceq\Phi$, a similar argument shows that $\mathcal{M}(\Phi_{Y_1})\subseteq\mathcal{M}$.

If $\Psi$ is a nontrivial template in the set $\{\Psi_1,\dots, \Psi_s\}$, a similar argument shows that either $\mathcal{M}^*(\Phi_{Y_0})\subseteq\mathcal{M}$, or $\mathcal{M}^*(\Phi_{Y_1})\subseteq\mathcal{M}$.
\end{proof}

\section{Growth Rates}
\label{Growth Rates}

Let $\mathcal{M}$ be a minor-closed class of matroids. Let $h_{\mathcal{M}}(r)$ denote the \textit{growth rate function} of $\mathcal{M}$: the function whose value at an integer $r\geq0$ is given by the maximum number of elements in a simple matroid in $\mathcal{M}$ of rank at most $r$. For a matroid $M$, we denote by $\varepsilon(M)$ the size of the simplification of $M$, that is the number of rank-1 flats of $M$. By combining the main result in ~\cite{gkw09} with earlier results of Geelen and Whittle ~\cite{gw03} and Geelen and Kabell ~\cite{gk09}, Geelen, Kung, and Whittle proved the following:

\begin{theorem}[Growth Rate Theorem]
\label{growthrate}

 If $\mathcal{M}$ is a nonempty minor-closed class of matroids, then there exists $c\in\mathbb{R}$ such that either:
\begin{itemize}
 \item[(1)] $h_{\mathcal{M}}(r)\leq cr$ for all $r$,
\item[(2)] $\binom{r+1}{2}\leq h_{\mathcal{M}}(r)\leq cr^2$ for all $r$ and $\mathcal{M}$ contains all graphic matroids,
\item[(3)] there is a prime-power $q$ such that $\frac{q^r-1}{q-1}\leq h_{\mathcal{M}}(r)\leq cq^r$ for all $r$ and $\mathcal{M}$ contains all $\mathrm{GF}(q)$-representable matroids, or
\item[(4)] $h_{\mathcal{M}}$ is infinite and $\mathcal{M}$ contains all simple rank-2 matroids.
\end{itemize}
\end{theorem}

If outcome (2) of the Growth Rate Theorem holds for a minor-closed class $\mathcal{M}$, then $\mathcal{M}$ is said to be \textit{quadratically dense}. In this section, we will consider growth rates of some quadratically dense classes of binary matroids. Let $\mathcal{EX}(F)$ denote the class of binary matroids with no $F$-minor. If $f$ and $g$ are functions, we write $f(r)\approx g(r)$ if $f(r)=g(r)$ for all but finitely many $r$.

Since the growth rate function for the class of graphic matroids is $\binom{r+1}{2}$, the Growth Rate Theorem implies that, if $F$ is a nongraphic binary matroid, \[h_{\mathcal{EX}(F)}(r)\geq\binom{r+1}{2}.\] Kung et. al. ~\cite{kmpr14} pose the following question: For which nongraphic binary matroids $F$ of rank 4 does equality hold above for all but finitely many $r$? Geelen and Nelson answer this question in ~\cite{gn15}. Let $N_{12}$ be the matroid formed by deleting a three-element independent set from $PG(3,2)$. The nongraphic binary matroids $F$ of rank 4 for which $h_{\mathcal{EX}(F)}(r)\approx\binom{r+1}{2}$ are exactly the nongraphic restrictions of $N_{12}$. We present here an alternate proof. Both proofs allow us to answer the question when $F$ is a matroid of any rank, not just rank 4. We will prove the following theorem after proving several lemmas.

\begin{theorem}
\label{quadgrowth}
 Let $\mathcal{M}$ be a minor-closed class of binary matroids. Then $h_{\mathcal{M}}(r)\approx\binom{r+1}{2}$ if and only if $\mathcal{M}$ contains all graphic matroids but does not contain $\mathcal{M}_v(\Phi_{Y_1})$.
\end{theorem}

Our proof of Theorem ~\ref{quadgrowth} will depend on the following theorem, proved by Geelen and Nelson in ~\cite{gn15}:
\begin{theorem}
\label{gn51}
Let $\mathcal{M}$ be a quadratically dense minor-closed class of matroids and let $p(x)$ be a real quadratic polynomial with positive leading coefficient. If $h_{\mathcal{M}}(n)>p(n)$ for infinitely many $n\in\mathbb{Z}^+$, then for all integers $r,s\geq1$ there exists a vertically $s$-connected matroid $M\in\mathcal{M}$ satisfying $\varepsilon(M)>p(r(M))$ and $r(M)\geq r$.
\end{theorem}

An \textit{even-cycle matroid} is a binary matroid of the form $M=M\binom{w}{D}$, where $D\in\mathrm{GF}(2)^{V\times E}$ is the vertex-edge incidence matrix of a graph $G=(V,E)$ and $w\in\mathrm{GF}(2)^E$ is the characteristic vector of a set $W\subseteq E$. The pair $(G,W)$ is an \textit{even-cycle representation} of $M$. The edges in $W$ are called \textit{odd} edges, and the other edges are \textit{even} edges. An \textit{odd cycle} of $(G,W)$ is a cycle of $G$ with an odd number of odd edges. A \textit{blocking pair} of $(G,W)$ is a pair of vertices $u,v$ of $G$ so that every odd cycle passes through at least one of these vertices. \textit{Resigning} at a vertex $u$ of $G$ occurs when all the edges incident with $u$ are changed from even to odd and vice-versa. It is easy to see that this corresponds to adding the row of the matrix corresponding to $u$ to the characteristic vector of $W$. Therefore, resigning at a vertex does not change an even-cycle matroid. It is also easy to see that if an even-cycle representation has a blocking pair, then we can resign so that every odd edge is incident with at least one vertex in the blocking pair. For our purposes, it will be convenient to think of a blocking pair in this way.

For $r\geq2$, let $A_r$ be the following binary matrix, where we choose for the $\Gamma$-frame matrix the matrix representation of $M(K_{r-1})$, so that the identity matrices are $(r-2)\times(r-2)$ matrices.
\begin{center}
\begin{tabular}{ |c|ccc|c|c|c|}
\hline
\multirow{2}{*}{0}&1&0&1&$1\cdots1$&$0\cdots0$&$1\cdots1$\\
&0&1&1&$0\cdots0$&$1\cdots1$&$1\cdots1$\\
\hline
$\Gamma$-frame matrix&\multicolumn{3}{c|}{0}&$I$&$I$&$I$\\
\hline
\end{tabular}
\end{center}
Note that $M(A_r)$ is the largest simple matroid of rank $r$ that virtually conforms to $\Phi_{Y_1}$.

\begin{definition}
Let $X_r$ be the largest simple matroid of rank $r$ that virtually conforms to $\Phi_{Y_1}$. Equivalently, $X_1=U_{1,1}$, and for $r\geq2$, we have $X_r=M(A_r)$.
\end{definition}

\begin{lemma}
 \label{Y1minors}
The class $\mathcal{M}_v(\Phi_{Y_1})$ is the class of matroids having an even-cycle representation with a blocking pair. This class is minor-closed.
\end{lemma}

\begin{proof}
 Any simple matroid $M$ virtually conforming to $\Phi_{Y_1}$ is a restriction of $X_r$ for some $r$.

Label the rows of $A_r$ as $1,\dots,r$. Add to the matrix row $r+1$, which is the sum of rows $2,\dots, r$. This does not change the matroid $X_r$. We see that $X_r$ is an even-cycle matroid $(G,W)$, where row 1 is the characteristic vector of $W$ and rows $2,\dots, r+1$ form the incidence matrix of $G$. Moreover, every edge in $W$ is incident with the vertex corresponding to either row 2 or row $r+1$. Thus, every matroid virtually conforming to $\Phi_{Y_1}$ has an even-cycle representation with a blocking pair. Conversely, every matroid that has an even-cycle representation with a blocking pair $\{u,v\}$ virtually conforms to $\Phi_{Y_1}$, by making $u$ correspond to the second row and making $v$ correspond to row $r+1$, which can be removed without changing the matroid.

By resigning whenever we wish to contract an element represented by an odd edge, it is not difficult to see that the class of matroids having an even-cycle representation with a blocking pair is minor-closed.
\end{proof}

\begin{lemma}
 \label{restriction}
 Any simple, rank-$r$ matroid $M$ that is a minor of a matroid virtually conforming to $\Phi_{Y_1}$ is a restriction of $X_r$.
\end{lemma}

\begin{proof}
 From the preceding lemma, $M$ is a restriction of some $X_{r'}$. So $M$ has an even-cycle representation $(G,W)$ with a blocking pair $\{u,v\}$. Let $w$ be the characteristic vector of $W$. There are $r'-r$ rows in the matrix $A_{r'}[(V\cup w)-{v}, E(M)]$ whose deletion does not alter the matroid $M$. After these rows are deleted, the resulting matrix is a submatrix of $A_r$.
\end{proof}

\begin{lemma}
\label{Y0minorY1}
Every matroid virtually conforming to $\Phi_{Y_1}$ is a minor of a matroid conforming to $\Phi_{Y_0}$.
\end{lemma}

\begin{proof}
By Lemma ~\ref{YCD}, we have $\Phi_{Y_1}\preceq\Phi_C$. Every matroid conforming to $\Phi_C$ is obtained by contracting an element from a matroid conforming to $\Phi_{Y_0}$.
\end{proof}

\begin{lemma}
\label{graphicvscographic}
 Let $k$ be a positive integer. Then there are at most finitely many integers $r$ such that the complete graphic matroid $M(K_{r+1})$ is a rank-($\leq k$) perturbation of a cographic matroid. 
\end{lemma}

\begin{proof}
Let $N$ be a cographic matroid. Observe that adding a rank-1 matrix to a matrix representation of a binary matroid $N$ changes $\varepsilon(N)$ by a factor of at most 2. This occurs when, in every rank-1 flat of $N$, there is at least one nonloop element indexing a column that is changed by adding the rank-1 matrix and at least one nonloop element indexing a column that remains unchanged when the rank-1 matrix is added. Thus, if $M$ is a rank-$(\leq t)$ perturbation of $N$, we have $\varepsilon(M)\leq2^t\varepsilon(N)$.

Let $r=r(M)$. Recall that a cographic matroid $N$ has $\varepsilon(N)\leq3r(N)-3$. Therefore, $\varepsilon(M)\leq2^t(3r(N)-3)\leq2^t(3(r+t)-3)$. For fixed $t$ and sufficiently large $r$, this expression is less than $\binom{r+1}{2}=\varepsilon(M(K_{r+1}))$.
\end{proof}

\begin{lemma}
\label{conformonly}
 Let $\mathcal{M}$ be a quadratically dense minor-closed class of matroids representable over a given field $\mathbb{F}$. Let $\{\Phi_1,\dots,\Phi_s,\Psi_1,\dots,\Psi_t\}$ be a set of templates describing $\mathcal{M}$. For sufficiently large $r$, the growth rate $h_{\mathcal{M}}(r)$ is equal to the size of the largest simple matroid of rank $r$ that virtually conforms to any template in $\{\Phi_1,\dots,\Phi_s\}$.
\end{lemma}

\begin{proof}
Let $h'_{\mathcal{M}}(r)$ denote the size of the largest simple matroid of rank $r$ that virtually conforms to any template in $\{\Phi_1,\dots,\Phi_s\}$. So $h_{\mathcal{M}}(r)\geq h'_{\mathcal{M}}(r)$. The size of the largest simple matroid of rank $r$ that virtually conforms to any particular template is a quadratic polynomial in $r$. Thus, for sufficiently large $r$, the function $h'_{\mathcal{M}}(r)$ is a quadratic polynomial as well.

By Definition ~\ref{describes}, there exist $k,l\in \mathbb{Z}_+$ so that every simple vertically $k$-connected member of $\mathcal{M}$ with at least $l$ elements either weakly conforms to a template in $\{\Phi_1,\dots,\Phi_s\}$ or weakly coconforms to some template in $\{\Psi_1,\dots,\Psi_t\}$. Suppose, for contradiction, that $h_{\mathcal{M}}(r)>h'_{\mathcal{M}}(r)$ for infinitely many $r$. Theorem ~\ref{gn51}, with $h'_{\mathcal{M}}(r)$ playing the role of $p(r)$, implies that there is a sequence $M_1, M_2,\dots$ of vertically $k$-connected matroids in $\mathcal{M}$ such that $\varepsilon(M_i)>h'_{\mathcal{M}}(i)$ and  $r(M_i)\geq i$. Thus, in this sequence, there are infinitely many matroids that are vertically $k$-connected and have size at least $l$. Since these matroids are too large to virtually conform to any template in $\{\Phi_1,\dots,\Phi_t\}$, there is at least one nontrivial template $\Psi\in\{\Psi_1,\dots,\Psi_t\}$ such that infinitely many vertically $k$-connected matroids in $\mathcal{M}$ coconform to $\Psi$. However, since $\mathcal{M}$ contains all graphic matroids and since every complete graphic matroid has infinite vertical connectivity (hence vertical $k$-connectivity), we have that infinitely many complete graphic matroids coconform to $\Psi$. For some $t$ depending on $\Psi$, every matroid coconforming to $\Psi$ is a rank-$(\leq t)$ perturbation of a cographic matroid. This contradicts Lemma ~\ref{graphicvscographic}. By contradiction, the result holds.
\end{proof}

\label{proof} 
\begin{proof}[Proof of Theorem ~\ref{quadgrowth}.]
First, suppose $h_{\mathcal{M}}(r)\approx\binom{r+1}{2}$. By the Growth Rate Theorem, $\mathcal{M}$ contains all graphic matroids. For $r\geq1$, we have $|X_r|=\binom{r-1}{2}+3r-3$, which for $r>2$ is greater than $\binom{r+1}{2}$. Thus, $\mathcal{M}$ does not contain $\mathcal{M}_v(\Phi_{Y_1})$.

Now, suppose $\mathcal{M}$ contains all graphic matroids but does not contain $\mathcal{M}_v(\Phi_{Y_1})$. Since $\mathcal{M}$ contains all graphic matroids, there is a nonempty set $\{\Phi_1,\dots,\Phi_s,\Psi_1,\dots,\Psi_t\}$ of binary frame templates describing $\mathcal{M}$. By Lemma ~\ref{conformonly}, $h_{\mathcal{M}}(r)$ is equal to the size of the largest simple matroid of rank $r$ that conforms to any template in $\{\Phi_1,\dots,\Phi_s\}$. Suppose $\Phi$ is a nontrivial template in $\{\Phi_1,\dots,\Phi_s\}$. By Corollary ~\ref{Y0Y1}, either $\Phi_{Y_0}\preceq\Phi$ or $\Phi_{Y_1}\preceq\Phi$. Since $\mathcal{M}$ does not contain $\mathcal{M}_v(\Phi_{Y_1})$, we must have $\Phi_{Y_0}\preceq\Phi$. However, by Lemma ~\ref{Y0minorY1}, this implies $\mathcal{M}_v(\Phi_{Y_1})\subseteq\mathcal{M}$. Therefore, we conclude that $h_{\mathcal{M}}(r)\approx\binom{r+1}{2}$, completing the proof.
\end{proof}

\begin{corollary}
\label{EXF}
Let $F$ be a simple, binary matroid of rank $r$. Then $h_{\mathcal{EX}(F)}\approx\binom{r+1}{2}$ if and only if $F$ is a nongraphic restriction of $X_r$.
\end{corollary}

\begin{proof}
By Theorem ~\ref{quadgrowth}, $h_{\mathcal{EX}(F)}\approx\binom{r+1}{2}$ if and only if $\mathcal{EX}(F)$ contains all graphic matroids but does not contain $\mathcal{M}_v(\Phi_{Y_1})$. The condition that $\mathcal{EX}(F)$ contains all graphic matroids is equivalent to the condition that $F$ is nongraphic.  By Lemma ~\ref{restriction}, the condition that $\mathcal{EX}(F)$ does not contain $\mathcal{M}_v(\Phi_{Y_1})$ is equivalent to the condition that $F$ is a restriction of $X_r$.
\end{proof}

Note that $X_4=N_{12}$; so this answers the question posed in ~\cite{kmpr14}.

We now consider the growth rate of $\mathcal{EX}(PG(3,2))$. We will prove Theorem ~\ref{EXPG32}, which we restate below.
\begin{theorem}
 The growth rate function for $\mathcal{EX}(PG(3,2))$ is \[h_{\mathcal{EX}(PG(3,2))}\approx r^2-r+1.\]
\end{theorem}

We will use the following.

\begin{lemma}
\label{PG32Phi}
Let $\mathcal{T}_{\mathcal{EX}(PG(3,2))}=\{\Phi_1,\dots\Phi_s,\Psi_1,\dots,\Psi_t\}$. If $\Phi\in\{\Phi_1,\dots\Phi_s\}$, then either $\Phi=\Phi_X$ or $\Phi$ is a template with $C=\emptyset$ and with $\Lambda$ and $\Delta$ trivial.
\end{lemma}

\begin{proof}
The class of matroids conforming to $\Phi_X$ is exactly the class of even-cycle matroids. This class is minor-closed. The largest simple, even-cycle matroid of rank $r$ has an even-cycle representation obtained from the graph $K_r$ by adding to each even edge an odd edge in parallel as well as adding one odd loop to the graph. Therefore, the class of even-cycle matroids has growth rate $2\binom{r}{2}+1=r^2-r+1$. So the largest simple, even-cycle matroid of rank 4 has size 13. Since $PG(3,2)$ has size $15$, we have $\mathcal{M}(\Phi_X)\subseteq\mathcal{EX}(PG(3,2))$. Therefore, we may assume that $\Phi_X\in\mathcal{T}_{\mathcal{EX}(PG(3,2))}$.

Since $\Phi_0\preceq\Phi_X$, we may assume that $\Phi_0\notin\{\Phi_1,\dots\Phi_s\}$. Let \[\Phi=(\{1\},C,X,Y_0,Y_1,A_1,\Delta,\Lambda)\] be a nontrivial template such that $\Phi\neq\Phi_X$ and $\Phi\in\{\Phi_1,\dots\Phi_s\}$. Consider the graft matroid $M(K_6,V(K_6))$. A straightforward computation shows that, by contracting the nongraphic element, we obtain $PG(3,2)$. Therefore, $\Phi_{Y_0}\npreceq\Phi$. By Lemma ~\ref{YCD}, we also have $\Phi_C\npreceq\Phi$ and $\Phi_{CX}\npreceq\Phi$.

Now, we may assume that $\Phi$ is in standard form. Since $\Phi_C\npreceq\Phi$, by Lemma ~\ref{PhiC} we may assume that $C_1=\emptyset$. Also, by Lemma ~\ref{PhiCD}, since $\Phi_{CX}\npreceq\Phi$ and $\Phi_C\npreceq\Phi$, either $\Lambda|X_1$ is nontrivial and $\Phi_X\preceq\Phi$ or $\Lambda$ is trivial and $C=\emptyset$.

First, suppose that $\Lambda$ is trivial and $C=\emptyset$. Since $\Phi_{Y_0}\npreceq\Phi$, Lemma ~\ref{PhiY0} implies that $\Phi$ is equivalent to a template with $\Delta$ trivial. So we may assume
\[\Phi=(\{1\},\emptyset,X,Y_0,Y_1,A_1,\{0\},\{0\}),\] which is one of the possible conclusions of the lemma.

Thus, we may assume that $\Lambda|X_1$ is nontrivial and $\Phi_X\preceq\Phi$. Suppose $|\Lambda|X_1|>2$. On the template \[\Phi=(\{1\},C_0,Y_0,Y_1,A_1,\Delta,\Lambda),\]
perform operation (3) and then repeatedly perform operations (4) and (10) to obtain the template
\[(\{1\},C_0,X,\emptyset,\emptyset,A_1[X,C_0],\{0\},\Lambda).\]
Then repeatedly perform operation (7) to obtain
\[(\{1\},\emptyset,X_1,\emptyset,\emptyset,[\emptyset],\{0\},\Lambda|X_1).\]

Since $\Lambda|X_1$ has characteristic 2 and size greater than 2, it contains a subgroup $\Lambda'$ isomorphic to $(\mathbb{Z}/2\mathbb{Z})\times(\mathbb{Z}/2\mathbb{Z})$. Perform operation (2) to obtain the template
\[(\{1\},\emptyset,X_1,\emptyset,\emptyset,[\emptyset],\{0\},\Lambda');\]
then repeatedly perform operations (5) and (6) to obtain
\[(\{1\},\emptyset,X',\emptyset,\emptyset,[\emptyset],\{0\},\Lambda''),\]
where $|X'|=2$ and $\Lambda''$ is the additive group generated by 
$\begin{bmatrix}
    1 \\
    0
\end{bmatrix}$ and $\begin{bmatrix}
    0 \\
    1
\end{bmatrix}$. One readily sees that $PG(3,2)$ conforms to this template. Therefore, $|\Lambda|=2$. We may perform row operations so that $\Lambda$ is generated by $[1,0\ldots,0]^T$. Let $\Sigma$ be the element of $X$ such that $\Lambda|\{\Sigma\}$ is nonzero.

Now, suppose there is an element $\bar{x}\in\Delta$ that is not in the row space of $A_1$. Perform operations (2) and (3) on $\Phi$ to obtain
\[(\{1\},C_0,X,Y_0,Y_1,A_1,\{0,\bar{x}\},\{0\}).\]
Now, by a similar argument to the one used in the proof of Lemma ~\ref{PhiY0}, we have $\Phi_{Y_0}\preceq\Phi$. Since we already know this is not the case, we deduce that every element of $\Delta$ is in the row space of $A_1$.

 Let $\bar{x}\in\Delta|C_0$ and $\bar{y}\in\Lambda$ be such that there are an odd number of natural numbers $i$ such that $\bar{x}_i=\bar{y}_i=1$. Then we call the ordered pair $(\bar{x},\bar{y})$ a \textit{pair of odd type}. Otherwise, $(\bar{x},\bar{y})$ is a \textit{pair of even type}. Suppose $(\bar{x},\bar{y})$ is a pair of odd type with $\bar{y}|X_1$ a zero vector. By performing operations (2) and (3) and repeatedly performing operations (4) and (10), we obtain
\[(\{1\},C_0,X,\emptyset,\emptyset,A_1[X,C],\{0,\bar{x}\},\{0,\bar{y}\}),\]
which is equivalent to $\Phi_{CX}$. We already know this is not the case. Therefore, for every pair $(\bar{x},\bar{y})$ of odd type, $\bar{y}|X_1=[1,0,\dots,0]^T$.

Suppose $\bar{x}\in\Delta|C$ and $\bar{y}_1,\bar{y}_2\in\Lambda$ are such that $\bar{y}_1|X_1=\bar{y}_2|X_1=[1,0,\dots,0]^T$, such that $(\bar{x},\bar{y}_1)$ is a pair of odd type, and such that $(\bar{x},\bar{y}_2)$ is a pair of even type. Then $(\bar{y}_1+\bar{y}_2)|X_1$ is a zero vector, and $(\bar{x},\bar{y}_1+\bar{y}_2)$ is a pair of odd type. Therefore, either all pairs $(\bar{x},\bar{y})\in\Delta|C\times\Lambda$ are of even type, in which case $\Phi$ is equivalent to a template with $\Lambda|X_0$ trivial and $C=\emptyset$, or if $(\bar{x},\bar{y})$ is a pair of odd type, then $(\bar{x},\bar{z})$ is of odd type for every $\bar{z}\in\Lambda$ with $\bar{z}|X_1$ nonzero. In this case, consider any matrix virtually conforming to $\Phi$. After contracting $C$, we can restore the $\Gamma$-frame matrix by adding $\Sigma$ to each row where the $\Gamma$-frame matrix has been altered. Therefore, $\Phi$ is equivalent to a template with $\Lambda|X_0$ trivial and $C=\emptyset$.

So we now have that
\[\Phi=(\{1\},\emptyset,X,Y_0,Y_1,A_1,\Delta,\Lambda),\]
with $\Lambda$ generated by $[1,0\ldots,0]^T$ and with every element of $\Delta$ in the row space of $A_1$. We will now show that, in fact, $\Phi$ is equivalent to a template with $\Delta$ trivial. On $\Phi$, perform $y$-shifts to obtain the following template, where $Y'_0=Y_0\cup Y_1$:
\[\Phi'=(\{1\},\emptyset,X,Y'_0,\emptyset,A_1,\Delta,\Lambda).\]
By repeatedly performing operation (5) and then operation (6) on this template, we may assume that $A_1$ has the following form, with the star representing an arbitrary binary matrix and $\bar{v}$ representing an arbitrary row vector:
\[
\left[
\begin{array}{c|c}
0\cdots0&\bar{v}\\
\hline
I_{|X|-1}&*
\end{array}
\right].
\]
Also, since $\Lambda|(D-\{\Sigma\})$ is trivial, we may perform row operations on every matrix conforming to $\Phi'$ to obtain a template
\[\Phi''=(\{1\},\emptyset,X,Y'_0,\emptyset,A_1,\Delta'',\Lambda),\]
so that every element of $\Delta''$ has 0 for its first $|X|-1$ entries. Since every element of $\Delta$ was in the row space of $A_1$, the only possible nonzero element of $\Delta''$ is the row vector with 0 for its first $|X|-1$ entries and whose last $|Y'_0|-|X|+1$ entries form the row vector $\bar{v}$. Note that operations (5) and (6) and the row operations we performed on every matrix conforming to $\Phi'$ each changes a template to an equivalent template. Thus, we may assume that $\bar{v}$ is nonzero and that $\Delta''=\{\bf{0},$$\bar{v}\}$ because otherwise, $\Phi$ is equivalent to a template with $\Delta$ trivial. So, for some $y\in Y'_0$, we have $\bar{v}_y=1$. On the template $\Phi''$, repeatedly perform operation (11) and then operation (10) to obtain the following template:
\[\Phi'''=(\{1\},\emptyset,\{\Sigma\},\{y\},\emptyset,[1],\mathbb{Z}/2\mathbb{Z},\mathbb{Z}/2\mathbb{Z}).\]

The following matrix conforms to $\Phi'''$:
\[
\left[
\begin{array}{ccccccccccccccc|c}
0&0&0&0&0&0&0&0&0&0&1&1&1&1&1&1\\
\hline
1&0&0&0&1&1&1&0&0&0&0&0&0&1&0&1\\
0&1&0&0&1&0&0&1&1&0&0&0&1&0&0&1\\
0&0&1&0&0&1&0&1&0&1&0&1&0&0&0&1\\
0&0&0&1&0&0&1&0&1&1&1&0&0&0&0&1
\end{array}
\right].
\]
By contracting $y$, we obtain $PG(3,2)$. Thus, we have shown that $\Phi$ must be equivalent to a template with $\Delta$ trivial. So we may assume
\[\Phi=(\{1\},\emptyset,X,Y_0,Y_1,A_1,\{0\},\Lambda),\]
with $\Lambda$ generated by $[1,0,\ldots,0]^T$.

Now, let us consider the structure of the matrix $A_1$. By repeated use of operation (5), we may assume that $A_1$ is of the following form, with the top row indexed by $\Sigma$, with $*$ representing an arbitrary row vector, with $Y_0=V_0\cup V_1$, and with each $L_i$ representing an arbitrary binary matrix:
\begin{center}
\begin{tabular}{ |c|c|c|c|c| }
\multicolumn{3}{c}{$Y_1$}&\multicolumn{1}{c}{$V_0$}&\multicolumn{1}{c}{$V_1$}\\
\hline
$0\cdots0$&$0\cdots0$&$1\cdots1$&$0\cdots0$&$*$\\
\hline
$I$&$L_0$&$L_1$&0&$L_2$\\
\hline
0&0&0&$I$&$L_3$\\
\hline
\end{tabular}
\end{center}

Suppose either $L_0$ or $L_1$ has a column with two or more nonzero entries. Let $y$ be the element of $Y_1$ that indexes that column, and let $Y'$ be the union of $\{y\}$ with the subset of $Y_1$ that indexes the columns of the identity submatrix of $A_1[X,Y_1]$. Repeatedly perform operations (4) and (10) on $\Phi$ to obtain
\[(\{1\},\emptyset,X,\emptyset,Y',A_1,\{0\},\Lambda).\]
On this template, repeatedly perform $y$-shifts, operation (11), and operation (6) to obtain 
\[(\{1\},\emptyset,X',\emptyset,Y'',\begin{bmatrix}
0&0&x\\
1& 0 &1\\
0& 1 & 1
\end{bmatrix},\{0\},\Lambda),\]
where $x=i$ if $y$ indexes a column of $L_i$ and where $X'$ and $Y''$ index the set of rows and columns, respectively, of the matrix $\begin{bmatrix}
0&0&x\\
1& 0 &1\\
0& 1 & 1
\end{bmatrix}$.

The following matrix conforms to this template. By contracting the columns printed in bold, we obtain $PG(3,2)$.
\[
\left[
\begin{array}{cccccccccccccc|cccc}
0&0&0&0&0&0&0&0&0&0&1&1&1&\bf{1}&\bf{0}&\bf{0}&x&x\\
0&0&0&0&0&0&0&0&0&0&0&0&0&\bf{0}&\bf{1}&\bf{0}&1&1\\
0&0&0&0&0&0&0&0&0&0&0&0&0&\bf{0}&\bf{0}&\bf{1}&1&1\\
\hline
1&0&0&0&1&1&1&0&0&0&0&0&0&\bf{1}&\bf{0}&\bf{0}&1&0\\
0&1&0&0&1&0&0&1&1&0&0&0&0&\bf{1}&\bf{0}&\bf{0}&0&1\\
0&0&1&0&0&1&0&1&0&1&1&0&1&\bf{0}&\bf{1}&\bf{0}&0&0\\
0&0&0&1&0&0&1&0&1&1&0&1&1&\bf{0}&\bf{0}&\bf{1}&0&0\\
\end{array}
\right].
\]
This shows that $L_0$ and $L_1$ consist entirely of unit and zero columns. Thus, by Lemma ~\ref{simpleY1}, $L_0$ is an empty matrix and $L_1$ consists entirely of distinct unit columns. Therefore, $A_1$ is of the following form:
\begin{center}
\begin{tabular}{ |c|c|c|c|c| }
\multicolumn{3}{c}{$Y_1$}&\multicolumn{1}{c}{$V_0$}&\multicolumn{1}{c}{$V_1$}\\
\hline
$0\cdots0$&$0\cdots0$&$1\cdots1$&$0\cdots0$&$*$\\
\hline
$I$&0&$I$&0&$Q_1$\\
\hline
0&$I$&0&0&$Q_2$\\
\hline
0&0&0&$I$&$Q_3$\\
\hline
\end{tabular}
\end{center}
with each $Q_i$ representing an arbitrary binary matrix.

Let $M$ be any matroid conforming to $\Phi$ with rank and connectivity functions $r$ and $\lambda$, respectively. Let $r'$ be the rank of the submatrix of $A_1$ consisting of $Q_1$, $Q_2$, and the row vector we have denoted with a star. Then $r(Y_0)=|V_0|+r'$ and $r(E(M)-Y_0)=r(M)-|V_0|$. Thus, $\lambda(Y_0)=r'$. So if $k>r'+1$, then $M$ is not vertically $k$-connected unless $Y_0$ or $E(M)-Y_0$ is spanning. If $Y_0$ is spanning in $M$, then the $\Gamma$-frame matrix used to construct $M$ has 0 rows. Thus, $M$ is not simple unless $|E(M)|\leq |Y_0|+|Y_1|+1$, with the 1 coming from the element $[1,0\cdots,0]^T$ of $\Lambda$. Thus, if we set $l>|Y_0|+|Y_1|+1$, then no simple, vertically $k$-connected matroid with at least $l$ elements conforms to $\Phi$ unless $E(M)-V_0$ is spanning in $M$. Therefore, we have $V_0=\emptyset$.

Let $Q$ be the submatrix of $A_1$ consisting of $Q_1$ and $Q_2$. If every column of $Q$ has at most two nonzero entries, then $\Phi\preceq\Phi_X$, and as we deduced above, we may assume $\Phi=\Phi_D$. Therefore, we assume that $Q$ has a column $c$, indexed by the element $y\in Y_0$ with three or more nonzero entries.

Repeatedly perform operation (10) on $\Phi$ to obtain the template
\[\Phi'=(\{1\},\emptyset,X,\{y\},Y_1,A_1[D,Y_1\cup\{y\}],\{0\},\Lambda).\]
 Let $c=\left[\begin{array}{c}
c_1\\
\hline
c_2
\end{array}\right]$, with $c_1$ a column of $Q_1$ and $c_2$ a column of $Q_2$. Consider the following cases:
\begin{enumerate}
\item[Case 1.] The vector $c_1$ has three nonzero entries.
\item[Case 2.] The vector $c_1$ has two nonzero entries, and $c_2$ has one nonzero entry.
\item[Case 3.] The vector $c_1$ has one nonzero entry, and $c_2$ has two nonzero entries.
\item[Case 4.] The vector $c_2$ has three nonzero entries.
\end{enumerate}

In Case $i$, repeatedly perform $y$-shifts and operation (11) to obtain the template
\[\Phi''_i=(\{1\},\emptyset,X',\{y\},Y'_1,A_{1,i},\{0\},\Lambda),\]
where $A_{1,i}$ is the matrix defined below with rows indexed by $X'$ and columns indexed by $Y'_1\cup\{y\}$. In each case, the last column is indexed by $y$, and it turns out that the value of $x$ does not matter.
\[
A_{1,1}=\left[
\begin{array}{ccc|ccc|c}
0&0&0&1&1&1&x\\
\hline
1&0&0&1&0&0&1\\
0&1&0&0&1&0&1\\
0&0&1&0&0&1&1\\
\end{array}
\right]
A_{1,2}=\left[
\begin{array}{ccc|cc|c}
0&0&0&1&1&x\\
\hline
1&0&0&1&0&1\\
0&1&0&0&1&1\\
0&0&1&0&0&1\\
\end{array}
\right]
\]

\[
A_{1,3}=\left[
\begin{array}{ccc|c|c}
0&0&0&1&x\\
\hline
1&0&0&1&1\\
0&1&0&0&1\\
0&0&1&0&1\\
\end{array}
\right]
A_{1,4}=\left[
\begin{array}{ccc|c}
0&0&0&x\\
\hline
1&0&0&1\\
0&1&0&1\\
0&0&1&1\\
\end{array}
\right]
\]

In Case $i$, the matrix below virtually conforms to $\Phi''_i$. By contracting the columns printed in bold, we obtain $PG(3,2)$.

\begin{enumerate}

\item[Case 1:]
\[
\left[
\begin{array}{ccc|cccccccccccc|c}
1&1&0&0&0&0&1&1&1&0&0&0&1&1&1&\textbf{\textit{x}}\\
0&0&0&1&0&0&1&0&0&1&0&0&1&0&0&\bf{1}\\
0&0&0&0&1&0&0&1&0&0&1&0&0&1&0&\bf{1}\\
0&0&0&0&0&1&0&0&1&0&0&1&0&0&1&\bf{1}\\
\hline
1&0&1&1&1&1&1&1&1&0&0&0&0&0&0&\bf{0}\\
\end{array}
\right]
\]

\item[Case 2:]
\[
\left[
\begin{array}{ccccccc|ccccccccc|c}
1&0&0&0&1&1&1&0&0&\bf{0}&1&1&0&0&1&1&\textbf{\textit{x}}\\
0&0&0&0&0&0&0&1&0&\bf{0}&1&0&1&0&1&0&\bf{1}\\
0&0&0&0&0&0&0&0&1&\bf{0}&0&1&0&1&0&1&\bf{1}\\
0&0&0&0&0&0&0&0&0&\bf{1}&0&0&0&0&0&0&\bf{1}\\
\hline
0&1&0&1&1&0&1&1&1&\bf{1}&1&1&0&0&0&0&\bf{0}\\
0&0&1&1&0&1&1&0&0&\bf{0}&0&0&1&1&1&1&\bf{0}\\
\end{array}
\right]
\]

\item[Case 3:]
\[
\left[
\begin{array}{ccccc|cccccccccccc|c}
0&0&0&\bf{1}&\bf{1}&0&0&0&0&0&0&1&0&0&0&0&0&\textbf{\textit{x}}\\
0&0&0&\bf{0}&\bf{0}&1&0&0&1&0&0&1&0&0&1&0&0&\bf{1}\\
0&0&0&\bf{0}&\bf{0}&0&1&0&0&1&0&0&1&0&0&1&0&\bf{1}\\
0&0&0&\bf{0}&\bf{0}&0&0&1&0&0&1&0&0&1&0&0&1&\bf{1}\\
\hline
1&0&1&\bf{1}&\bf{0}&1&1&1&0&0&0&0&0&0&0&0&0&\bf{0}\\
0&1&1&\bf{0}&\bf{1}&0&0&0&1&1&1&1&0&0&0&0&0&\bf{0}\\
0&0&0&\bf{0}&\bf{1}&0&0&0&0&0&0&0&1&1&0&0&0&\bf{0}\\
\end{array}
\right]
\]

\item[Case 4:]
\[
\left[
\begin{array}{ccccc|cccccccccccc|c}
1&0&\bf{1}&1&\bf{1}&0&0&0&0&0&0&0&0&0&0&0&0&\textbf{\textit{x}}\\
0&0&\bf{0}&0&\bf{0}&1&0&0&1&0&0&1&0&0&1&0&0&\bf{1}\\
0&0&\bf{0}&0&\bf{0}&0&1&0&0&1&0&0&1&0&0&1&0&\bf{1}\\
0&0&\bf{0}&0&\bf{0}&0&0&1&0&0&1&0&0&1&0&0&1&\bf{1}\\
\hline
0&0&\bf{1}&0&\bf{0}&1&1&1&0&0&0&0&0&0&0&0&0&\bf{0}\\
0&0&\bf{0}&0&\bf{1}&0&0&0&1&1&1&0&0&0&0&0&0&\bf{0}\\
0&1&\bf{0}&1&\bf{1}&0&0&0&0&0&0&1&1&1&0&0&0&\bf{0}\\
\end{array}
\right]
\]
\end{enumerate}
By contradiction, this completes the proof.
\end{proof}

\label{proof} 
\begin{proof}[Proof of Theorem ~\ref{EXPG32}.]
Let $\mathcal{M}=\mathcal{EX}(PG(3,2))$, and let $\mathcal{T}_{\mathcal{M}}=\linebreak\{\Phi_1,\dots\Phi_s,\Psi_1,\dots,\Psi_t\}$. By Lemma ~\ref{conformonly}, for sufficiently large $r$, we have $h_{\mathcal{M}}(r)$ equal to the size of the largest simple matroid of rank $r$ that virtually conforms to any template in $\Phi\in\{\Phi_1,\dots\Phi_s\}$. If $\Phi\in\{\Phi_1,\dots\Phi_s\}$, then by Lemma ~\ref{PG32Phi} either $\Phi=\Phi_X$ or $\Phi$ is of the form $(\{1\},\emptyset,X,Y_0,Y_1,A_1,\{0\},\{0\})$, for some matrix $A_1$ and some sets $X$, $Y_0$, and $Y_1$. Moreover, by operation (5), we may assume that $A_1$ is of the following form, with $Y_0=V_0\cup V_1$ and with the stars representing arbitrary binary matrices:

\begin{center}
\begin{tabular}{ |c|c|c|c| }
\multicolumn{2}{c}{$Y_1$}&\multicolumn{1}{c}{$V_0$}&\multicolumn{1}{c}{$V_1$}\\
\hline
$I$&$*$&0&$*$\\
\hline
0&0&$I$&$*$\\
\hline
\end{tabular}
.
\end{center}

The largest simple matroid of rank $r$ that virtually conforms to $\Phi$ is obtained by taking for the $\Gamma$-frame matrix a matrix representation of $M(K_{n+1})$, where $n=r-r(M(A_1[X,Y_1]))-|V_0|$. Thus, the largest simple matroid of rank $r$ that virtually conforms to $\Phi$ has size $\binom{n+1}{2}+|Y_1|n+|Y_1|+|Y_0|$. Substituting $r-r(M(A_1[X,Y_1]))-|V_0|$ for $n$, one sees that for sufficiently large $r$, this expression is less than $r^2-r+1$. Since the class of matroids virtually conforming to $\Phi_X$ is the class of even-cycle matroids, which has growth rate $r^2-r+1$, the result holds.
\end{proof}

\section{1-flowing Matroids}
\label{1-flowing Matroids}

In this section, we prove Theorem ~\ref{1flowing}. The 1-flowing property is a generalization of the max-flow min-cut property of graphs. See Seymour ~\cite{s81} or Mayhew ~\cite{m15} for more of the background and motivation concerning 1-flowing matroids. We follow the notation and exposition of ~\cite{m15}.

\begin{definition}
Let $e$ be an element of a matroid $M$. Let $c_x$ be a non-negative integral capacity assigned to each element $x\in E(M)-e$. A flow is a function $f$ that assigns to each circuit $C$ containing $e$ a non-negative real number $f_C$ with the constraint that for each $x\in E-e$, the sum of $f_C$ over all circuits containing both $e$ and $x$ is at most $c_x$. We say that $M$ is \textit{$e$-flowing} if, for every assignment of capacities, there is a flow whose sum over all circuits containing $e$ is equal to \[\min\{\sum_{x\in C^*-e}c_x | C^* \textnormal{ is a cocircuit containing }e \}.\] If $M$ is $e$-flowing for each $e\in E(M)$, then $M$ is \textit{1-flowing}.
\end{definition}

The matroid $T_{11}$ is the even-cycle matroid obtained from $K_5$ by adding a loop and making every edge odd, including the loop. In ~\cite{s81}, Seymour showed the following.
\begin{proposition}
 The The class of 1-flowing matroids is minor-closed. Moreover, $AG(3,2)$, $U_{2,4}$, $T_{11}$, and $T^*_{11}$ are excluded minors for the class of 1-flowing matroids.
\end{proposition}

Seymour ~\cite{s81} conjectured that these are the only excluded minors.
\begin{conjecture}[Seymour's 1-flowing Conjecture]
 The set of excluded minors for the class of 1-flowing matroids consists of $AG(3,2)$, $U_{2,4}$, $T_{11}$, and $T^*_{11}$.
\end{conjecture}

Since $U_{2,4}$ is an excluded minor for the class of 1-flowing matroids, all such matroids are binary. Therefore, the results in this paper apply to 1-flowing matroids. We will now prove Theorem ~\ref{1flowing}, which we restate below.
\begin{theorem}
 There exist $k,l\in\mathbb{Z}_+$ such that every simple, vertically $k$-connected, 1-flowing matroid with at least $l$ elements is either graphic or cographic.
\end{theorem}

\begin{proof}
 The matroid $AG(3,2)$ conforms to $\Phi_{Y_1}$ since it is a restriction of $N_{12}$. Indeed, consider the matrix representing $N_{12}$ that virtually conforms to $\Phi_{Y_1}$. Add the rows labeled by $X$ in this matrix to one of the other rows. Then we can see the matrix representation $[I_4|J_4-I_4]$ of $AG(3,2)$ as a restriction of $N_{12}$. Also, it is not difficult to see that $AG(3,2)$ can be obtained from a matroid conforming to $\Phi_{Y_0}$ by contracting $Y_0$. Thus, $\mathcal{EX}(AG(3,2))$ contains neither $\mathcal{M}(\Phi_{Y_0})$ nor $\mathcal{M}(\Phi_{Y_1})$. Since $AG(3,2)$ is self-dual, $\mathcal{EX}(AG(3,2))$ does not contain $\mathcal{M}^*(\Phi_{Y_0})$, or $\mathcal{M}^*(\Phi_{Y_1})$ either. Therefore, by Corollary ~\ref{Y0Y1}, $\mathcal{EX}(AG(3,2))$ is described by the trivial template. Thus, since $AG(3,2)$ is an excluded minor for the class of 1-flowing matroids, there exist $k,l\in\mathbb{Z}_+$ such that every simple, vertically $k$-connected, 1-flowing matroid with at least $l$ elements either conforms or coconforms to the trivial template. The result follows.
\end{proof}

\section*{Acknowledgements}
We thank the two anonymous referees for carefully reading the manuscript. In particular, we thank the first referee for giving many suggestions that improved the manuscript, including one that greatly simplified the proof of Lemma ~\ref{graphicvscographic}.

\end{document}